\documentclass[12pt]{article}

\usepackage{amssymb,amsmath,amsthm,amsfonts,color}
\usepackage[colorlinks=true,urlcolor=blue,linkcolor=blue,citecolor=blue,pdfstartview=FitH]{hyperref}
\usepackage{authblk}
\usepackage[sort,comma]{natbib}

\begin{document}

\newtheorem{thm}{Theorem}[section]
\newtheorem{lem}[thm]{Lemma}
\newtheorem{cnd}[thm]{Condition}
\newtheorem{prop}[thm]{Proposition}
\newtheorem{rem}[thm]{Remark}
\newtheorem{hyp}[thm]{Hypothesis}
\newtheorem{coro}[thm]{Corollary}
\newtheorem{example}[thm]{Example}
\newtheorem{defn}[thm]{Definition}

\newcommand{\thmref}[1]{Theorem~{\rm \ref{#1}}}
\newcommand{\lemref}[1]{Lemma~{\rm \ref{#1}}}
\newcommand{\corref}[1]{Corollary~{\rm \ref{#1}}}
\newcommand{\cndref}[1]{Condition~{\rm \ref{#1}}}
\newcommand{\propref}[1]{Proposition~{\rm \ref{#1}}}
\newcommand{\defref}[1]{Definition~{\rm \ref{#1}}}
\newcommand{\remref}[1]{Remark~{\rm \ref{#1}}}
\newcommand{\exmref}[1]{Example~{\rm \ref{#1}}}
\newcommand{\figref}[1]{Figure~{\rm \ref{#1}}}
\newcommand{\asmpref}[1]{Assumption~{\rm \ref{#1}}}
\newcommand{\sectref}[1]{Section~{\rm \ref{#1}}}

\newcommand{\EE}{\mathbb E}
\newcommand{\NN}{\mathbb N}
\newcommand{\RR}{\mathbb R}

\renewcommand {\theequation}{\arabic{section}.\arabic{equation}}
\def \non{{\nonumber}}
\def \Bbb{{\sf}}
\def \hat{\widehat}
\def \bar{\overline}

\newcommand{\pbgf}{\phi \beta \gamma f}

%\bigskip

\newcommand {\R} {\mathbb{R}}

\renewcommand {\theequation}{\arabic{section}.\arabic{equation}}
%\widenspacing

% Command \mlabel: puts label names in the right margin for
% easy reference while doing early drafts of a document.
%
\newcommand{\mlabel}[1]{\label{#1}\mshow{#1}}
\newcommand{\mshow}[1]{%
 \ifmmode
  \global\edef\lname{#1}%
  \aftergroup\mshow\aftergroup\lname
 \else
  \ifinner
   \global\edef\lname{#1}%
   \aftergroup\mshow\aftergroup\lname
  \else
   \marginpar{\footnotesize #1}%
  \fi
 \fi}

\begin{center}
{\large %Working Paper on \\ 
%\medskip

{\bf Linear Programming Formulations of\\
\medskip

Singular Stochastic Control Problems: \\
\medskip

Time-Homogeneous Problems}}
\bigskip

{\sc Thomas G. Kurtz\footnote{Departments of Mathematics and Statistics,
University of Wisconsin - Madison, 480 Lincoln Drive, Madison, WI 53706-1388,
kurtz@math.wisc.edu.}
and Richard H. Stockbridge\footnote{Department of Mathematical Sciences, University of Wisconsin Milwaukee,
Milwaukee, WI  53201--0413, stockbri@uwm.edu.  This research was supported in part by the Simons Foundation under grant award 246271.}}\\
%\medskip

%20 August 2008
%\medskip
\end{center}

{\footnotesize
{\bf Abstract.}  Conditions are established under which the optimal control of processes having both absolutely continuous and singular (with respect to time) controls are equivalent to linear programs over a space of measures on the state and control spaces.  This paper considers long-term average and discounted criteria and includes budget and resource constraints.  The linear programs optimize over measures representing the expected occupation measure of the state and absolutely continuous control processes and a similar expected occupation measure of the state and control when the singular action of the process occurs.  The evolution of these processes is characterized through an adjoint equation which the measures must satisfy in relation to the absolutely continuous and singular generators of the process.  Existence of optimal relaxed controls of feedback type are established in general while existence of an optimal form of strict control is proven under additional closedness and compactness conditions.

\medskip


\medskip

{\bf Abbreviated Title.} Linear programming for time-homogeneous singular control.

}

\setcounter{equation}{0}

\newcommand {\doma} {\mbox{${\cal D}$ }}
\newcommand {\chat} {\mbox{$\hat{C}(E)$} }
\newcommand {\cbar} {\mbox{$\overline{C}(E \times U)$} }
\newcommand {\bofe} {\mbox{${\cal B}(E)$ }}
\newcommand {\bofu} {\mbox{${\cal B}(U)$ }}
\renewcommand {\hat}{\widehat}

\setcounter{equation}{0}

\section{Introduction} \label{sec1}
Consider processes whose dynamics are specified through a singular, controlled martingale problem for their generators, that is, by the requirement that
\begin{equation} \label{scmp}
f(X(t)) - \int_0^t Af(X(s),u(s))\, ds - \int_0^t Bf(X(s),u(s))\, d\xi(s)
\end{equation}
be a martingale for every $f$ in the common domain ${\cal D}$ of the operators $A$ and $B$.  In this expression, $X$ is the state process, $u$ is a process which controls $X$, $A$ is the  generator of $X$, $B$ is the generator that captures the singular behavior and $\xi$ is an increasing process whose collection of times of increase is typically singular with respect to Lebesgue measure.  (A relaxed formulation of the dynamics is provided in subsection~1.1.)  The decision maker chooses the controls so as to optimize some prescribed criterion.  

This paper considers long-term average and discounted criteria while a companion paper addresses first exit, finite horizon and optimal stopping criteria.  It establishes conditions under which the original control problems are equivalent to linear programs over measures representing the expected occupation of the state and (absolutely continuous) control processes and the expected occupation of the state and control when the singular behavior of the process occurs.  The evolution of these processes is characterized through an adjoint equation which the measures must satisfy in relation to the generators related to the process.

The equivalence of linear programs and discrete stochastic control problems was observed by \cite{mann:60} and has been used for Markov decision processes, for example by \cite{hern:96,hern:98,hern:99}.  Equivalence for absolutely continuous stochastic control problems in continuous time has been established under a long-term average criterion by \cite{stoc:90b}, though the optimal relaxed control was not shown to be of feedback type.  \cite{bhat:96} and \cite{kurt:98} improved this equivalence by establishing the feedback form of an optimal relaxed control for long-term average problems, and extending the results to discounted, finite-horizon and first exit criteria.  Equivalence of optimal stopping problems with linear programs has been shown for absolutely continuous processes by \cite{cho:02}; this result was extended to include singular behavior of the process by \cite{helm:07}.  \cite{taks:97} considered singular control of diffusion under a discounted criterion.  Using infinite-dimensional linear programming methods, he formulated the linear program (similarly to Section~\ref{secdisc} below) and its dual and showed the absence of a duality gap.  He related the dual linear program to a quasivariational inequality associated with the stochastic problem.  Our two companion papers extend the results of \cite{kurt:98} to include singular behavior and/or controls based on an existence result in \cite{kurt:01} and applies to five typical decision criteria.  They also extend the results of \cite{helm:07} to include control of the process before stopping.  In addition, these papers allow budget and resource constraints of the same type as the decision criterion, which were not included in \cite{kurt:98}.  

For compactness reasons, \cite{bhat:96}, \cite{kurt:98} and \cite{kurt:01} use a relaxed formulation of controls as probability measures on the control space.  The question naturally arises as to what conditions imply the existence of an optimal control in the class of strict controls.  \cite{haus:90} provided one set of conditions for controlled diffusions in the presence of budget and resource constraints while \cite{dufo:12} used the linear programming equivalences of \cite{kurt:98} to adapt the conditions of Haussmann and Lepeltier to the more general setting of processes which are solutions of controlled martingale problems.  The current paper further extends these results to include processes having singular behavior and control.

This paper considers time-homogeneous models under long-term average and discounted criteria.  The results characterize optimal stationary relaxed controls in feedback form.  The companion paper addresses models under first exit, finite horizon, and optimal stopping criteria in which the controls are naturally time-dependent.  It also provides a characterization of optimal time-dependent relaxed controls of feedback type.  

The current paper is organized as follows.  The formulation of the time-homogeneous processes, the budget and resource constraints and optimality criteria are given in Section~\ref{formulation} while  the statement of the existence result in \cite{kurt:01} is presented in Section~\ref{sec:existence}. The proof of this result in \cite{kurt:01} has a minor error; the correction is given in Section~\ref{sec:existence}.  Section~\ref{subsect:technical} contains two technical results that are used later in the paper as well as in its companion.  The paper then considers the two cost criteria in turn.  Section~\ref{seclta} considers the long-term average control problem and proves the equivalence of the linear program.  The discounted problem is examined in Section~\ref{secdisc} in which it is first necessary to establish an existence result corresponding to a discounted form of the adjoint relation before proving equivalence between the discounted problem and the discounted form of the linear program.  Both equivalence results establish existence of an optimal pair of measures for the linear program and characterize optimal absolutely continuous and singular relaxed controls of feedback type using these  measures.  Section~\ref{sect:strict} gives sufficient closedness and compactness conditions for the existence of optimal strict feedback controls while Section~\ref{sect:example} illustrates the equivalence results for the inventory problem below.

We conclude this section with two examples of singularly controlled processes.  A canonical example will also be given at the end of Section~\ref{formulation}.

\begin{example} \label{finite-fuel-ex}
Consider the finite fuel follower problem of \cite{bene:80}.  The controlled process $X$ satisfies
$$X(t) = x_0 + W(t) - \xi(t), \qquad t \geq 0,$$
in which $W$ is a standard Brownian motion process and the process $\xi$ satisfies $\xi(0)=0$ and $\int_0^\infty d|\xi|(t) \leq y$ where $|\xi|$ denotes the total variation of $\xi$.  The quantity $y$ represents the amount of fuel available.  The decision criterion is
$$J_\alpha(\xi;x_0) := \EE\left[\int_0^\infty e^{-\alpha t} X^2(t)\, dt\right].$$
Note that the decision criterion has $c_0(x) = x^2$ so only depends on the state of the process; no cost is accrued for controlling the process.

%A simple modification of the problem would impose a charge for control.  Breaking the process $\xi$ into an absolutely continuous (with respect to Lebesgue measure of time) process $\xi_{ac}$ and a singular process $\xi_s$, the dynamics can be written as
%$$X(t) = x_0 + W(t) - \xi_{ac}(t) - \xi_s(t) = x_0 + W(t) - \int_0^t u(r)\, dr - \xi_s(t), \qquad t\geq 0.$$
%The cost structure might combine a quadratic running cost in both state and control with a cost rate that is proportional to the amount of control applied:
%$$J_\alpha(u,\xi_s;x_0) = \EE\left[\int_0^\infty e^{-\alpha r} (X^2(r) + u^2(r))\, dr + \int_0^\infty e^{-\alpha r} c_1\, d|\xi_s|(r)\right].$$
\end{example}

The second example presents an inventory control problem.  When the ordering cost includes a fixed cost, this problem is in the class of impulse control problems.

\begin{example}[\cite{86:sule}] \label{inventory-example}
The inventory level process $X$ satisfies
$$X(t) = x_0 -\mu t + \sigma\, W(t) + \sum_{k=1}^\infty I_{\{\tau_k \leq t\}} Y_k$$
in which $\mu, \sigma > 0$, $W$ is a standard Brownian motion process and $(\tau,Y):=\{(\tau_k,Y_k): k \in \NN\}$ is an admissible ordering policy.  To be admissible, each $\tau_k$ must be a stopping time relative to $\{{\cal F}^W_t\}$, the filtration generated by $W$, and each $Y_k$ must be non-negative and ${\cal F}_{\tau_k}$-measurable.  

The cost structure includes holding/back-order running costs and fixed plus proportional ordering costs.  Specifically, define
\begin{equation} \label{inv-run-cost}
c_0(x) = \left\{\begin{array}{rl}
-c_b\, x, & \quad x < 0, \\
 c_h\, x, & \quad x \geq 0,
\end{array}\right.
\end{equation}
in which $c_b > 0$ denotes the back-order cost rate per unit of inventory per unit of time and similarly, $c_h > 0$ is the holding cost rate.  Also let $k_1 > 0$ denote the fixed cost and $k_2$ denote the cost per unit ordered.  The discounted cost criterion for the inventory problem is
\begin{equation} \label{inv-disc-cost}
J_\alpha(\tau,Y) = \EE\left[\int_0^\infty e^{-\alpha t} c_0(X(t))\, dt + \sum_{k=1}^\infty I_{\{\tau_k < \infty\}} (k_1 + k_2 Y_k)\right],
\end{equation}
in which $\alpha > 0$ is the discount factor, while the long-term average criterion is
\begin{equation} \label{inv-lta-cost}
J_0(\tau,Y) = \limsup_{t\rightarrow \infty} \mbox{$\frac{1}{t}$} \EE\left[\int_0^t c_0(X(s))\, ds + \sum_{k=1}^\infty I_{\{\tau_k\leq t\}} (k_1 + k_2 Y_k)\right].
\end{equation}
\end{example}

\subsection{Formulation of singular control problems} \label{formulation} 
For the dynamics of the processes, we employ the formulation of \cite{kurt:01} and refer the reader to that paper for a discussion of pre-generators (which are used in condition \ref{gencnd}).

For a complete, separable, metric space $S$, we define $M(S)$ to 
be the space of Borel measurable functions on $S$, $B(S)$ to 
be the space of bounded, measurable functions on $S$, $C(S)$ 
to be the space of continuous functions on $S$, $\overline C(S)$ to be 
the space of bounded, continuous functions on $S$, $\widehat{C}(S)$ to be
the space of continuous functions vanishing at $\infty$, ${\cal M}(S)$ to 
be the space of finite Borel measures on $S$, and ${\cal P}(S)$ to 
be the space of probability measures on $S$.  ${\cal M}(S)$ and 
${\cal P}(S)$ are topologized by weak convergence.

Let ${\cal L}_t(S)={\cal M}(S\times [0,t])$.  We define ${\cal L}(S)$ to be the space of measures $\xi$ on $S\times [0,\infty )$ such that $\xi (S\times [0,t])<\infty$, for each $t$, and topologized so that $\xi_n\rightarrow\xi$ if and only if $\int fd\xi_n\rightarrow\int fd\xi$, for every $f\in\overline C(S\times [0,\infty ))$ with supp$(f)\subset S\times [0,t_f]$ for some $t_f<\infty$.  Let $\xi_t\in {\cal L}_t(S)$ denote the restriction of $\xi$ to $S\times [0,t]$.  Note that a sequence $\{\xi^n\}\subset {\cal L}(S)$ converges to a $\xi\in {\cal L}(S)$ if and only if there exists a sequence $\{t_k\}$, with $t_k\rightarrow \infty$, such that, for each $t_k$, $\xi^n_{t_k}$ converges weakly to $\xi_{t_k}$, which in turn implies $\xi^n_t$ converges weakly to $\xi_t$ for each $t$ satisfying $\xi (S\times \{t\})=0$. 

Throughout, we will assume that the state space $E$ and control space $U$ are complete, separable, metric spaces.
\medskip

\noindent
{\em Dynamics.}\/  Let $A,B:{\cal D}\subset\bar {C}(E)\rightarrow C(E\times U)$ be linear operators and  $\nu_0\in {\cal P}(E)$.  Let $(X,\Lambda )$ be an $E\times {\cal P}(U)$-valued process and $\Gamma$ be an ${\cal L}(E\times U)$-valued random variable.  Let $\Gamma_t$ denote the  restriction of $\Gamma$ to $E\times U\times [0,t]$.    Then $(X,\Lambda ,\Gamma )$ is a relaxed solution of the {\em singular, controlled martingale problem }\/ for $(A,B,\nu_0)$ if there exists a filtration $\{{\cal F}_t\}$ such that $(X,\Lambda ,\Gamma_t)$ is $\{{\cal F}_t\}$-progressive, $X(0)$ has distribution $\nu_0$, and for every $f\in {\cal D}$, 
\begin{equation} \label{mgp}
f(X(t))-f(X(0))-\int_0^t\int_UAf(X(s),u)\Lambda_s(du)ds-\int_{E\times U\times [0,t]}Bf(x,u)\Gamma (dx\times du\times ds)
\end{equation}
is an $\{{\cal F}_t\}$-martingale.  Note we allow relaxed controls (controls represented by probability distributions on $U$) and a relaxed formulation of the singular part.
\medskip

Rather than require all control values $u\in U$ to be available for every state $x\in E$, we allow the availability of controls to depend on the state.  Let ${\cal U}\subset E\times U$ be a closed set, and define 
\[U_x=\{u:(x,u)\in {\cal U}\}.\]
Let $(X,\Lambda ,\Gamma )$ be a solution of the singular, controlled martingale problem for $(A,B,\nu_0)$.  The control $\Lambda$ and the random measure $\Gamma$ are {\em admissible\/} if for every $t$, 
\begin{eqnarray} \label{lamadmis}
&&\int_0^tI_{{\cal U}}(X(s),u)\Lambda_s(du)ds=t,\mbox{\rm \ and}\ \\
\label{gamadmis}
&&\Gamma ({\cal U}\times [0,t])=\Gamma (E\times U \times [0,t]).
\end{eqnarray}
Note that condition (\ref{lamadmis}) essentially requires $\Lambda_s$ to have support in $U_x$ when $X(s)=x$.

This paper examines time-homogeneous models.  We assume that the absolutely continuous generator $A$ and the singular generator $B$ have the following properties.  

\begin{cnd}\label{gencnd}
\begin{itemize}
\item[(i)] $A,B:{\cal D}\subset\overline C(E)\rightarrow C(E\times U)$, $1\in {\cal D}$,
and $A1 = 0, B1 = 0$.

\item[(ii)] There exist $\psi_A, \psi_B\in C(E \times U)$, 
$\psi_A, \psi_B \geq 1$, and constants $ a_f, b_f$, 
$f\in {\cal D}$,
such that
\[|Af(x,u)| \leq a_f\psi_A(x,u),\qquad |Bf(x,u)| \leq b_f\psi_B(x,u),\qquad \forall (x,u) \in {\cal U}.\]

\item[(iii)] There exists a countable collection $\{f_k\} \subset {\cal D}$ such that (\ref{mgp}) being a martingale for all $f_k$ implies (\ref{mgp}) is a martingale for all $f \in {\cal D}$ so that $(X,\Lambda,\Gamma)$ is a solution of the singular, controlled martingale problem for $(A,B,\nu_0)$.
%Defining $(A_0,B_0)=\{(f,\psi_A^{-1}Af,\psi_B^{-1}Bf): f \in {\cal D}\}$, $(A_0,B_0)$  is {\em separable\/} in the sense that there exists a countable collection $\{g_k\}\subset {\cal D}$ such that $(A_0,B_0)$ is contained in the bounded, pointwise closure of the linear span of $\{(g_k,A_0g_k,B_0g_k) = (g_k,\psi_A^{-1}Ag_k,\psi_B^{-1}Bg_k)\}$.

\item[(iv)] For each $u\in U$, the operators $A_u$ and $B_u$ defined by 
$A_uf(x)=Af(x,u)$ and $B_uf(x) = Bf(x,u)$ are pre-generators.

\item[(v)] ${\cal D}$ is closed under multiplication and separates points.
\end{itemize}
\end{cnd}

\begin{rem}
Condition~\ref{gencnd}(iv) is quite general.  For example, if $E$ is compact, $A: C(E)\rightarrow C(E)$, and $A$ satisfies the positive maximum principle, then $A$ is a pre-generator.  If $E$ is locally compact, $A:\hat{C}(E)\rightarrow \hat{C}(E)$, and $A$ satisfies the positive maximum principle, then $A$ can be extended to a pre-generator on $E^{\Delta}$, the one-point compactification of $E$.  We refer the reader to \cite{kurt:01} for further explanation and examples.
\end{rem}

\begin{rem} 
\begin{itemize}
\item[(i)] Condition~\ref{gencnd}(ii) is used to obtain bounded functions $A_\psi f = Af/\psi$ and $B_\psi f = Bf/\psi$ for each $f \in {\cal D}$ and for compactness criteria on the space $U$ of controls.  Different conditions may also be sufficient to obtain the results.  In particular, when both the state and control spaces are compact, the condition is trivially satisfied.  
\item[(ii)] The separability requirement of Condition~\ref{gencnd}(iii) is used within the framework of a complete, separable metric space to compactify the state space.  This condition can be avoided when the state space is compact and can be replaced by a simpler condition when $E$ is locally compact and ${\cal D}\subset \widehat{C}(E)$.  The reader is referred to \cite{kurt:98} for the latter condition.
\end{itemize}
\end{rem}
\medskip

\noindent
{\em Decision Criteria.}\/  For simplicity of notation, we denote the relaxed controls by the pair $(\Lambda,\Gamma)$ for a solution $(X,\Lambda,\Gamma)$ of the singular, controlled martingale problem for $(A,B,\nu_0)$; the relaxed singular control is more properly given by a transition function $\eta(x,s,du)$ satisfying $\Gamma(G_1\times G_2) = \int_{G_1}\eta(x,s,G_2)\, \Gamma(dx \times U\times ds)$ for $G_1\in {\cal B}(E\times [0,\infty))$ and $G_2\in {\cal B}(U)$.  

To compare controls, we consider two standard criteria, namely, the long-term average cost
\begin{equation} \label{ltacost}
J_0(\Lambda,\Gamma) = \begin{array}[t]{l} \displaystyle \limsup_{t\rightarrow \infty} t^{-1} \EE\left[\int_0^t \int_U c_0(X(s),u)\, \Lambda_s(du) ds \right. \\ \displaystyle \left.
\qquad \qquad \qquad + \int_{E \times U \times [0,t]} c_1(x,u) \Gamma(dx \times du \times ds)\right] \end{array}
\end{equation}
and the discounted cost
\begin{equation} \label{disccost}
J_\alpha(\Lambda,\Gamma;\nu_0) = \begin{array}[t]{l} \displaystyle
\EE\left[\int_0^\infty \int_U e^{-\alpha s} c_0(X(s),u)\, \Lambda_s(du) ds \right. \\ \displaystyle 
\left. \qquad \qquad + \int_{E \times U \times [0,\infty)} e^{-\alpha s} c_1(x,u)\, \Gamma(dx \times du \times ds)\right]
\end{array}
\end{equation}
in which $c_0$ is the cost rate related to the absolutely continuous evolution and control of the process and $c_1$ is the cost arising from the singular actions and behavior.
\medskip

\noindent
{\em Budget and Resource Constraints.}\/  \cite{haus:90} discusses how to represent hard (a.s.) constraints as soft (in mean) constraints by allowing the functions to take value $\infty$.  We therefore express these additional constraints as soft constraints.  

For $m<\infty$ and $i = 1,\ldots, m$, let $g_i,h_i: E\times U \rightarrow \RR^+$ be lower semicontinuous and bounded below, and $0 <  K_i < \infty$.  We allow additional constraints of the same form as the decision criterion.  Thus for control problems having the long-term average criterion $J_0(\Lambda,\Gamma)$, these constraints require 
\begin{equation} \label{lta-budget-constrs}
\begin{array}{r} \displaystyle
\limsup_{t\rightarrow \infty} \mbox{$\frac{1}{t}$} \EE\left[\int_0^t g_i(X(s),u)\, \Lambda_s(du)\, ds + \int_{E\times U\times [0,t]} h_i(x,u)\, \Gamma(dx\times du\times ds)\right] \leq K_i, \\
i=1,\ldots,m,
\end{array}
\end{equation}
while for problems using the discounted cost criterion $J_\alpha(\Lambda,\Gamma)$, the additional restrictions are 
\begin{equation} \label{disc-budget-constrs}
\begin{array}{r} \displaystyle 
\EE\left[\int_0^\infty e^{-\alpha s} g_i(X(s),u)\, \Lambda_s(du)\, ds + \int_{E\times U\times [0,\infty)} e^{-\alpha s} h_i(x,u)\, \Gamma(dx\times du\times ds)\right] \leq K_i, \\
i = 1, \ldots, m.
\end{array}
\end{equation}

We place additional conditions on the cost and budget functions and on the singular generator $B$.  First, note that a function $c: S \rightarrow \RR$ is inf-compact if for each $a>0$, the set $\{s: c(s) \leq a\}$ is compact; in this $S$ is a topological space.

\begin{cnd} \label{costcnd1}
\begin{itemize}
\item[(a)] The cost functions $c_0$ and $c_1$  are non-negative, lower semi-continuous and inf-compact.
\item[(b)] For some positive constants $a_0,b_0,a_1,b_1$ and $0 \leq \beta < 1$,
$$\psi_A(x,u) \leq a_0 c_0^\beta(x,u) + b_0,\qquad \psi_B(x,u) \leq a_1 c_1(x,u) + b_1, \qquad \forall (x,u) \in E \times U.$$ 
\item[(c)] Either the singular cost function $c_1$ or a singular budget function $h_i$ is positive and bounded away from $0$, for some $i \in \{1, \ldots, m\}$. 
\item[(d)] Either $Bf$ is bounded for every $f\in {\cal D}$, or there exists a compactification $\overline{E\times U}$ of $E\times U$ such that both (i) and (ii) hold: 
\begin{itemize}
\item[(i)] for $(x,u) \in \overline{E\times U} - E\times U$, defining 
$$\frac{c_1(x,u)}{\psi_B(x,u)} = \limsup_{\mbox{\scriptsize$\begin{array}{c} (y,v)\rightarrow (x,u) \\ (y,v) \in E\times U\end{array}$}} \frac{c_1(y,v)}{\psi_B(y,v)},$$ 
the function $c_1/\psi_B$ on $\overline{E\times U}$ is lower semi-continuous;  
\item[(ii)] for each $f \in {\cal D}$, $Bf/\psi_B$ has a continuous extension to $\overline{E\times U}$.
\end{itemize}
\end{itemize}
\end{cnd}

We conclude this subsection with an example of a controlled process satisfying the above conditions.

\begin{example}
A canonical example would be to take $A,B_k\subset C(E)\times C(E)$ to be generators for  processes on a compact state space $E$ (e.g., 
$E={\RR}^d\cup \{\infty \}$) with common domain ${\cal D}$, $U=[0,\infty )^m$,  $\bar {U}=U\cup\partial U$ to be the compactification in which $\partial U=\{z:z_k\geq 0,\sum_{k=1}^mz_k=1\}$ and $u^n\in U\rightarrow z$ if $\sum_k u_k^n\rightarrow\infty$ and $\frac 1{\sum_{k=1}^m u_k^n} u^n \rightarrow z$ in ${\RR}^m$, and let
\[Af(x,u)=Af(x),\quad \mbox{ and } \quad Bf(x,u)=\sum_{k=1}^mu_kB_kf(x).\]
Take
\[c_1(x,u)=\sum_{k=1}^mu_k\beta_k(x),\]
where the $\beta_k$ are continuous and strictly positive, and
\[\psi_B(x,u)=\sum_{k=1}^mu_k+1.\]
Then, under modest assumptions on $c_0$, Condition 1.5 will be satisfied.
\end{example}

\subsection{Existence of stationary solutions.} \label{sec:existence}
The equivalence of the linear programming and stochastic control formulations of the problems
rely on the existence of stationary solutions to the singular, controlled martingale problems
established in \cite{kurt:01}.  The stationary solutions are then used to define new
solutions appropriate for the optimality criterion under consideration.  

For completeness of exposition, we state the existence results.  In addition, the statement is correct, but the proof given in \cite{kurt:01} has a small error.  We provide the correction to the proof in this paper.

We say that an ${\cal L}(E)$-valued random variable has 
stationary increments if for $a_i<b_i$, $i=1,\ldots ,m$,
the distribution of 
$(\Gamma (H_1\times (t+a_1,t+b_1]),\ldots ,\Gamma (H_m\times (t+a_
m,t+b_m]))$ does not 
depend on $t$.  Let $X$ be a measurable stochastic process 
defined on a complete probability space $(\Omega ,{\cal F},P)$, and let 
${\cal N}\subset {\cal F}$ be the collection of null sets.
Then ${\cal F}_t^X=\sigma (X(s):s\leq t)$, $\bar {{\cal F}}^X_t={\cal N}\vee {\cal F}_t^X$ will denote the 
completion of ${\cal F}_t^X$, and $\bar {{\cal F}}^X_{t+}=\cap_{s>t}\bar {{\cal F}}^X_s$.   Let $E_1$ and 
$E_2$ be complete, separable metric spaces.  $q:E_1\times {\cal B}(E_2)\rightarrow [0,1]$ is a 
{\em transition function\/} from $E_1$ to $E_2$ if for each $x\in E_1$, 
$q(x,\cdot )$ is a Borel probability measure on $E_2$, and for each 
$D\in {\cal B}(E_2)$, $q(\cdot ,D)\in B(E_1)$.  If $E=E_1=E_2$, then we say 
that $q$ be a transition function on $E$.

\begin{thm}\ \label{thm:stationary-existence}
Let $A$, $B$, $\psi_A$ and $\psi_B$ satisfy Condition~\ref{gencnd}.
Suppose that $\mu_0\in {\cal P}(E\times U)$ and $\mu_1\in {\cal M}(E\times U)$ satisfy 
\begin{equation}
\mu_0({\cal U})=\mu_0(E\times U)=1,\qquad \mu_1({\cal U})=\mu_1(E\times U)< \infty,\label{muadmis0}
\end{equation}

\begin{equation}\label{integrability}
\int\psi_A(x,u)\mu_0(dx\times du) + \int\psi_B(x,u)\mu_1(dx\times du)<\infty ,
\end{equation}
and
\begin{equation} \label{id.1}
\int_{E\times U}Af(x,u)\,\mu_0(dx\times du) + \int_{E\times U}Bf(x,u)\,\mu_1(dx\times du)=0,\qquad \forall f
\in {\cal D}.
\end{equation}
For $i=0,1$, let $\mu_i^E$ be the state marginal of $\mu_i$ and 
let $\eta_i$ be the transition function from $E$ to $U$ such that
\begin{equation} \label{etadefs}
\mu_i(dx\times du)=\eta_i(x,du)\mu_i^E(dx).
\end{equation}

Then there exist a process $X$ and a random measure $\Gamma$ on $E\times [0,\infty )$, adapted to $\{\bar {{\cal F}}_{t+}^X\}$, such that:
\begin{description}
\item[$\circ$] $X$ is stationary and $X(t)$ has distribution $\mu^E_0$;
\item[$\circ$] $\Gamma$ has stationary increments, 
$\Gamma (E\times [0,t])$ is finite for each $t$, and
$\EE[\Gamma (\cdot\times [0,t])]=t\mu^E_1(\cdot )$; and
\item[$\circ$] For each $f\in {\cal D}$,
\begin{eqnarray} \label{mg2}\nonumber
f(X(t)) - f(X(0))&-&\int_0^t\int_UAf(X(s),u)\,\eta_0(X(s),du)ds\\
&-&\int_{E\times [0,t]}\int_UBf(y,u)\,\eta_1(y,du)\,\Gamma(dy\times ds)
\end{eqnarray}
is an $\{\bar {{\cal F}}^X_{t+}\}$-martingale.
\end{description}
\end{thm}

\begin{rem}
Observe that $\Gamma$ is adapted to $\{\bar {{\cal F}}_{t+}^{\;X}\}$ though the definition of the solution of a 
singular, controlled martingale problem did not require it. 
\end{rem}

\begin{proof}
The proof of this theorem in \cite{kurt:01} consists of constructing the desired process $X$ and random measure $\Gamma$ as limits of approximating quantities.  The measure constructed, however, only satisfies the inequality
$$\EE[\Gamma(\cdot \times [0,t])] \leq t \mu_1^E(\cdot)$$
rather than the claimed equality.

Recalling that $\Gamma$ has stationary increments, define $\mu_1^*$ so that
$$\mu_1^*(C) t = \EE[\Gamma(C \times [0,t])], \qquad C \in {\cal B}(E).$$
$\Gamma$ as constructed does satisfy 
$$\EE\left[\int_{E\times [0,t]} \int_U Bf(x,u)\eta_1(x,du)\, \Gamma(dx \times ds)\right] = \int_0^t \int_E \int_U Bf(x,u) \eta_1(x,du)\, \mu_1^E(dx)\, ds, \;\; f \in {\cal D}$$
and hence, setting $\tilde{\mu}_1(C) = \mu_1^E(C) - \mu_1^*(C)$, we have
\begin{equation} \label{mu1-id1}
\int_E \int_U Bf(x,u) \eta_1(x,du)\, \tilde{\mu}_1(dx) = 0, \qquad f \in {\cal D}.
\end{equation}
Consequently, denoting the measure constructed in the proof in \cite{kurt:01} as $\Gamma_1$, rather than $\Gamma$, and setting $\Gamma = \Gamma_1 + \tilde{\mu}_1\times m$, where $m$ denotes Lebesgue measure on $[0,\infty)$, (\ref{mg2}) is a martingale (adding $\tilde{\mu}_1 \times m$ does not change the process) and the newly defined $\Gamma$ does satisfy $\EE[\Gamma(\cdot \times [0,t])] = t \mu_1^E(\cdot)$.

In the case of constrained processes, there frequently exists $\varphi \in {\cal D}$ such that $B\varphi(x,u) > 0$, for $(x,u) \in {\cal U}$.  Then (\ref{mu1-id1}) implies $\tilde{\mu}_1 = 0$ and the original construction and the statements of other results are correct.  In the control setting, $\mu_0$ and $\mu_1$ are frequently chosen to minimize an expression of the form
$$\int_{E\times U} c_0(x,u)\, \mu_0(dx\times du) + \int_{E\times U} c_1(x,u)\, \mu_1(dx\times du)$$
subject to (\ref{id.1}), and if $c_1(x,u) > 0$, for $(x,u) \in {\cal U}$, it again follows that $\tilde{\mu}_1 = 0$.
\end{proof}

Theorem \ref{thm:stationary-existence} can in turn be used to extend 
the results in the Markov setting to operators with 
range in $M(E\times U)$, that is, we relax the continuity assumptions of earlier results.  We state the
result but retain the continuity assumptions in the sequel.

\begin{coro}\label{echev2}
Let $E$ and $U$ be complete, separable metric spaces.
Let $\hat{A},\hat{B}:{\cal D}\subset\overline C(E)\rightarrow M(E\times U)$, 
and suppose  $\hat{\mu}_0\in {\cal P}(E\times U)$ 
and $\hat{\mu}_1\in {\cal M}(E\times U)$ satisfy
\begin{equation}
\int_{E\times U} \hat{A}f(x,u)\,\hat{\mu}_0(dx\times du) 
+ \int_{E\times U} \hat{B}f(x,u)\,\hat{\mu}_1(dx \times du)=0,\qquad \forall f\in {\cal D}.
\label{invariant}
\end{equation}
Assume that there exist a complete, separable, metric space $V$, functions $\psi_A$ and $\psi_B$ and operators  $A,B:{\cal D}\rightarrow C(E\times U \times V)$ satisfying Condition \ref{gencnd}, and transition functions $\eta_0$ and $\eta_1$ from $E\times U$ to $V$
such that 
\[\hat{A}f(x,u)=\int_V Af(x,u,v)\,\eta_0(x,u,dv),\quad\hat{B}f(x,u)=\int_V Bf(x,u,v)\,\eta_1(x,u,dv),
\quad\forall f\in {\cal D},\]
and 
\[\int_{E\times U \times V}\psi_A(x,u,v)\,\eta_0(x,u,dv)\hat{\mu}_0(dx\times du) 
+ \int_{E\times U \times V}\psi_B(x,u,v)\,\eta_1(x,u,dv)\hat{\mu}_1(dx\times du)<\infty .\]
Then there exists a solution 
$(X,\hat{\Lambda},\hat{\Gamma})=(X,\hat{\eta}_0(X,\cdot),\hat{\eta}_1\,\Gamma )$ 
of the singular martingale problem 
for $(\hat {A},\hat {B},\hat{\mu}_0)$, where $\hat{\eta}_i$ satisfies 
$\hat{\mu}_i(dx \times du) = \hat{\eta}_i(x,du) \hat{\mu}_i^E(dx)$,
such that $X$ is stationary and $\Gamma$ has stationary increments.
\end{coro}

\begin{proof}  Define 
$$\mu_0(dx\times du\times dv)=\eta_0(x,u,dv)\hat{\mu}_0(dx\times du) = \eta_0(x,u, dv) \hat{\eta}_0(x,du)
\hat{\mu}_0^E(dx) $$
and 
$$\mu_1(dx\times du\times dv)=\eta_1(x,u,dv)\hat{\mu}_1(dx\times du)= \eta_1(x,u, dv) \hat{\eta}_1(x,du)
\hat{\mu}_1^E(dx).$$  The 
corollary follows immediately from Theorem~\ref{thm:stationary-existence}.  
\end{proof}

Though Theorem~\ref{thm:stationary-existence} establishes the existence of a stationary solution $(X,\Lambda,\Gamma)$ of the singular, controlled martingale problem for $(A,B)$ corresponding to any pair $(\mu_0,\mu_1)$ satisfying (\ref{id.1}), it is still necessary to address the budget and resource constraints.  

\begin{coro} \label{cor:lta-budget-constrs-meas}
Let $A$, $B$, $\psi_A$, $\psi_B$, $\mu_0$ and $\mu_1$ satisfy the hypotheses of Theorem~\ref{thm:stationary-existence} and let $X$ be the resulting stationary process and $\Gamma$ be the resulting random measure having stationary increments.  Define the relaxed control $\Lambda$ by $\Lambda_s(\cdot) = \eta_0(X(s),\cdot)$ and the $E\times U \times [0,\infty)$-valued random variable $\widetilde\Gamma$ such that 
$$\widetilde\Gamma(G_1 \times G_2 \times [0,t]) = \int_{G_1\times [0,t]} \eta_1(x,G_2)\, \Gamma(dx\times ds), \qquad G_1\in {\cal B}(E), G_2\in {\cal B}(U).$$
Then the solution $(X,\Lambda,\widetilde\Gamma)$ of the singular, controlled martingale problem for $(A,B)$ satisfies the budget and resource constraints (\ref{lta-budget-constrs}) if and only if
\begin{equation} \label{meas-lta-budget-constrs}
\int g_i(x,u)\, \mu_0(dx \times du) + \int h_i(x,u)\, \mu_1(dx\times du) \leq K_i,\qquad i=1,\ldots,m.
\end{equation}
\end{coro}

\begin{proof}
This immediately follows from the stationarity of $(X,\Gamma)$ with $X(t)$ having distribution $\mu_0^E$ and $\EE[\Gamma(\cdot\times [0,t])] = t\mu_1^E(\cdot)$ and the use of $\eta_0$ and $\eta_1$ in defining the relaxed controls.
\end{proof}

\subsection{Preliminary technical results} \label{subsect:technical}
We now add a technical lemma concerning random measures having stationary increments which will be used several times in the sequel and in the companion paper. 

\begin{lem} \label{statmeas1}
Let $\Psi$ be a random measure on $E \times [0,\infty)$ such that $\Psi$ has stationary increments,
$\Psi(E \times [0,t])$ is finite for each $t$ and $\EE(\Psi(\cdot \times [0,t]) = t \mu(\cdot)$
for some $\mu \in {\cal M}(E)$.  Let $h$ be a bounded, continuous function on $E$.
Let $\{\tau_k\}$ be a sequence of random variables such that,
defining $k^t_0 = \max\{k: \tau_k < t\}$, $k^t_1 = \min\{k: \tau_k \geq t\}$ and $k^t_2 = k^t_1+1$, 
$$(\tau_{k^t_1}-\tau_{k^t_0})^{-1} \int_{E \times [\tau_{k^t_1},\tau_{k^t_2})} h(x) \Psi(dx \times ds)$$
is stationary as a process in $t$ and for each
$k$, $\EE[\int_{E \times [\tau_k,\tau_{k+1})}h(x)\, \Psi(dx \times ds)] < \infty$.  Then
$$\EE\left[(\tau_{k^t_1}-\tau_{k^t_0})^{-1} \int_{E \times [\tau_{k^t_1},\tau_{k^t_2})} h(x) \Psi(dx \times ds)\right]
= \int h(x)\, \mu(dx).$$
\end{lem}

\begin{proof}
Assume for simplicity of notation that when $t=0$, $k^t_0 = 0$, $k^t_1 = 1$ and $k^t_2 = 2$.  For $t \geq 0$,
let $N(t)$ denote the number of $\tau_k$s taking values in $[0,t]$.  %(Note that the notational assumption implies
%$N(t) = k^t_0$.)  
By stationarity,
\begin{eqnarray*}
\lefteqn{\EE\left[(\tau_{k^t_1}-\tau_{k^t_0})^{-1} \int_{E \times [\tau_{k^t_1},\tau_{k^t_2})} h(x) \Psi(dx \times ds)\right]}
\\
&=& T^{-1} \int_0^T 
\EE\left[(\tau_{k^t_1}-\tau_{k^t_0})^{-1} \int_{E \times [\tau_{k^t_1},\tau_{k^t_2})} h(x) \Psi(dx \times ds)\right]\, dt \\
&=& T^{-1} \EE\left[\sum_{k=1}^{N(T)} \frac{(\tau_{k+1} \wedge T) - (\tau_k \vee 0)}{\tau_{k+1}-\tau_k}
\int_{E\times [\tau_{k+1},\tau_{k+2})} h(x)\, \Psi(dx \times ds) \right] \\
&=& T^{-1} \EE\left[\int_{E\times [0,T)} h(x)\, \Psi(dx \times ds) \right] \\
& & - T^{-1} 
\EE\left[\left(1 - \frac{\tau_1 \wedge T}{\tau_1-\tau_0}\right)\int_{E\times [0,\tau_1\wedge T)} h(x)\, \Psi(dx \times ds) 
\right] \\
& & + T^{-1} \EE\left[I_{\{N(T)=1\}}\, \frac{\tau_1}{\tau_1-\tau_0}\int_{E\times [T,\tau_2)} h(x)\, \Psi(dx \times ds) 
\right] \\
& & + T^{-1} \EE\left[I_{\{N(T)\geq 2\}} \int_{E\times [T,\tau_{N(T)+1})} h(x)\, \Psi(dx \times ds) 
\right] \\
& & + T^{-1} \EE\left[I_{\{N(T) \geq 1\}} \left(\frac{T - \tau_{N(T)}}{\tau_{N(T)+1}-\tau_{N(T)}}\right)
\int_{E \times [\tau_{N(T)+1},\tau_{N(T)+2})} h(x)\, \Psi(dx \times ds)\right].
\end{eqnarray*}
The first term of the right-hand-side equals $\int h(x)\, \mu(dx)$ and, as $T\rightarrow \infty$, the other terms
converge to $0$.
\end{proof}

We conclude this section with a final proposition concerning the existence of limits of feasible pairs $\{(\mu_0^n,\mu_1^n): n \in \NN\}$ for the adjoint relation (\ref{id.1}) and budget constraints (\ref{meas-lta-budget-constrs}).  

\begin{prop} \label{feasible-limits}
Assume Conditions~\ref{gencnd} and \ref{costcnd1} hold.  For each $n\in \NN$, suppose that $(\mu_0^n,\mu_1^n) \in {\cal P}(E\times U) \times {\cal M}(E\times U)$ satisfies (\ref{muadmis0}), (\ref{integrability}), (\ref{id.1}) and (\ref{meas-lta-budget-constrs}), with 
\begin{equation} \label{cost-limsup-finite}
\limsup_{n\rightarrow \infty} \left(\int c_0\, d\mu_0^n + \int c_1\, d\mu_1^n\right) = C < \infty.
\end{equation} 
Then there exists a pair $(\mu_0,\mu_1) \in  {\cal P}(E\times U) \times {\cal M}(E\times U)$ satisfying the adjoint relation (\ref{id.1}) and the budget constraints (\ref{meas-lta-budget-constrs}) for which 
$$\int c_0\, d\mu_0 + \int c_1\, d\mu_1 \leq C.$$
\end{prop}

\begin{proof}
We first show that (\ref{cost-limsup-finite}) implies the tightness of the measures $\{(\mu_0^n,\mu_1^n): n\in \NN\}$.  Let $\epsilon > 0$ be chosen arbitrarily and pick $M > (C+1)/\epsilon$.  Recall, both $c_0$ and $c_1$ are non-negative.  Define the compact set $K=\{(x,u)\in E\times U: c_0(x,u) \vee c_1(x,u) \leq M\}$.  Let $N$ be large enough such that for all $n \geq N$, 
\begin{equation} \label{n-large}
\int c_0(x,u)\, \mu_0^n(dx\times du) + \int c_1(x,u)\, \mu_1^n(dx\times du) \leq C + 1.
\end{equation}
Then for each $n \geq N$,
\begin{equation} \label{lta-tightness}
\begin{array}{rcl} 
\mu_0^n(K^c) + \mu_1^n(K^c) &\leq& \displaystyle \int_{K^c} \frac{c_0(x,u)}{M}\, \mu_0^n(dx\times du) + \int_{K^c} \frac{c_1(x,u)}{M}\, \mu_1^n(dx\times du) \\
&\leq& \displaystyle \mbox{$\frac{1}{M}$}\left(\int c_0(x,u)\, \mu_0^n(dx\times du) + \int c_0(x,u)\, \mu_0^n(dx\times du) \right) \\
&\leq& \mbox{$\frac{C+1}{M}$} < \epsilon.
\end{array}
\end{equation}
For the finitely many $n$ for which $n < N$, one can find compact sets $K_n$ such that $\mu_0^n(K_n^c) + \mu_1^n(K_n^c) < \epsilon$ and hence taking the union of these compact sets establishes the tightness of both $\{\mu_0^n\}$ and $\{\mu_1^n\}$.  

Let $(\mu_0,\mu_1)$ be a limit of $\{(\mu_0^n,\mu_1^n)\}$.  If $c_1$ satisfies Condition~\ref{costcnd1}(c) with lower bound $a > 0$, then for $n \geq N$, $\mu_1^n(E\times U) \leq \frac{C+1}{a}$ which implies $\mu_1 \in {\cal M}(E\times U)$.  Similarly, if $h_i$ satisfies Condition~\ref{costcnd1}(c) with positive lower bound $a$, then $\mu_1^n(E\times U) \leq \frac{K_i}{a}$, again establishing that $\mu_1$ is a finite measure.

We further claim that $(\mu_0,\mu_1)$ satisfies the adjoint relation (\ref{id.1}).  Let $\{n_k\}$ be a subsequence such that $(\mu_0^{n_k},\mu_1^{n_k})\Rightarrow (\mu_0,\mu_1)$ as $k\rightarrow \infty$.  Arbitrarily pick $f \in {\cal D}$ and observe that for each $k$, 
$$\int Af(x,u)\, \mu_0^{n_k}(dx\times du) + \int Bf(x,u)\, \mu_1^{n_k}(dx\times du) = 0.$$

Begin by examining the convergence related to the first integral.  By Condition~\ref{gencnd}(ii), $\frac{|Af|}{\psi_A} \leq a_f$.  Now for each $k$, define the measure $\hat{\mu}_0^{n_k}$ to have Radon-Nikodym derivative $\psi_A$ relative to $\mu_0^{n_k}$.  We show that $\{\hat{\mu}_0^{n_k}: k\in \NN\}$ is relatively compact.  Choose $\epsilon > 0$ arbitrarily, pick $M$ such that $M^{1-\beta} > \frac{(a_0+b_0)(C+1)}{\epsilon} \vee 1$ and define the compact set $K_1 = \{(x,u)\in E\times U: c_0(x,u) \leq M\}$.  By Condition~\ref{costcnd1}(c), on the set $K_1^c$, 
$$\psi_A \leq a_0 c_0^\beta + b_0 = \mbox{$\frac{c_0}{M^{1-\beta}} \cdot \left(a_0 + \frac{b_0}{c^\beta}\right) \cdot \frac{M^{1-\beta}}{c_0^{1-\beta}} < \frac{a_0+b_0}{M^{1-\beta}}\; c_0.$}$$
Thus, letting $N$ be sufficiently large that (\ref{n-large}) holds for $n_k \geq N$, it follows that for such $n_k$, 
\begin{eqnarray*}
\hat{\mu}_0^{n_k}(K^c) = \int_{K^c} \psi_A(x,u)\, \mu_0^{n_k}(dx\times du) &\leq& \mbox{$\frac{a_0+b_0}{M^{1-\beta}}$} \int c_0(x,u)\, \mu_0^{n_k}(dx\times du) \\ 
&\leq& \mbox{$\frac{(a_0+b_0)(C+1)}{M^{1-\beta}} < \epsilon$}.
\end{eqnarray*}
As a result, $\{\hat{\mu}_0^{n_k}\}$ is tight so there exists some subsequence $\{n_{k_\ell}\}$ and a limiting measure $\hat{\mu}_0$ such that $\hat{\mu}_0^{n_{k_\ell}} \Rightarrow \hat{\mu}_0$ as $k_\ell\rightarrow \infty$.  Note that for any bounded continuous function $h$, the fact that $\psi_A \geq 1$ along with weak convergence implies
\begin{equation} \label{separating}
\int h(x,u)\, \mu_0(dx\times du) = \int \frac{h(x,u)}{\psi_A(x,u)}\, \hat{\mu}_0(dx\times du)
\end{equation}
and hence $\hat{\mu}_0$ has Radon-Nikodym derivative $\psi_A$ with respect to $\mu_0$.  Since $\frac{Af}{\psi_A}$ is bounded and continuous, the desired convergence also follows.

Now consider the simple case in Condition~\ref{costcnd1}(d) of $Bf$ being bounded.  Then weak convergence immediately yields $\int Bf\, d\mu_1^{n_k} \rightarrow \int Bf\, d\mu_1$ and (\ref{id.1}) follows.  When the second option in Condition~\ref{costcnd1}(d) holds so $Bf/\psi_B$ extends continuously to $\overline{E\times U}$, define the measures $\hat{\mu}_1^{n_k}$ on $\overline{E\times U}$ to have Radon-Nikodym derivative $\psi_B$ with respect to $\mu_1^{n_k}$ for $k \in \NN$.  Clearly $\{\hat{\mu}_1^{n_k}: k \in \NN\}$ is tight and thus relatively compact.  Thus there exists a subsequence of $\{n_{k_\ell}\}$ and limiting measure $\hat{\mu}_1$ so that applying the same argument as for (\ref{separating}), it follows that $\hat{\mu}_1$ has Radon-Nikodym derivative $\psi_B$ with respect to $\mu_1$.  By Condition~\ref{gencnd}(ii) and using the continuous extension of $Bf/\psi_B$, 

$$\lim_{k_\ell \rightarrow \infty} \int Bf\, d\mu_1^{n_{k_\ell}} = \lim_{k_\ell \rightarrow \infty} \int \frac{Bf}{\psi_B}\, d\hat{\mu}^{n_{k_\ell}}_1 = \int \frac{Bf}{\psi_B}\, d\hat{\mu}_1 = \int Bf\, d\mu_1$$
and again (\ref{id.1}) follows.  More precisely, the adjoint relation under this second option is
$$\int_{E\times U} Af(x,u)\, \mu_0(dx\times du) + \int_{\overline{E\times U}} Bf(x,u)\, \mu_1(dx\times du) = 0, \qquad \forall f \in {\cal D}.$$

Turning to an analysis of the cost, applications of the Skorohod representation theorem and Fatou's lemma implies 
$$\int c_0\, d\mu_0 + \int c_1\, d\mu_1 \leq \liminf_{k_\ell \rightarrow \infty} \left(\int c_0\, d\mu_0^{n_{k_\ell}} + \int c_1\, d\mu_1^{n_{k_\ell}}\right) \leq C$$
and for each $i = 1,\ldots,m$,
$$\int g_i\, d\mu_0 + \int h_i\, d\mu_1 \leq \liminf_{k_\ell \rightarrow \infty} \left(\int g_i\, d\mu_0^{n_{k_\ell}} + \int h+i\, d\mu_1^{n_{k_\ell}}\right) \leq K_i.$$
\end{proof}

\setcounter{equation}{0}

\section{Long-term average control problem.} \label{seclta}
Theorem~\ref{thm:stationary-existence} shows existence of a stationary process $X$ and random measure $\Gamma$ having stationary increments corresponding to measures $\mu_0$ and $\mu_1$ satisfying (\ref{id.1}) while Corollary~\ref{cor:lta-budget-constrs-meas} characterizes the budget constraints.  These results can be used directly in establishing an equivalent linear program for the long-term average control problem.  

\begin{thm} \label{thm:lta-lp}
Assume Conditions~\ref{gencnd} and \ref{costcnd1} hold.  Then the problem of minimizing the long-term average cost $J_0(\Lambda,\Gamma)$ in (\ref{ltacost}) over relaxed solutions $(X,\Lambda,\Gamma)$ of the singular, controlled martingale problem for $(A,B)$ that satisfy the budget constraints (\ref{lta-budget-constrs}) is equivalent to the linear program 
\begin{equation} \label{lta-lp}
\begin{array}{ll}
\mbox{Minimize} & \displaystyle \int_{E\times U} c_0(x,u)\, \mu_0(dx\times du) + \int_{E\times U} c_1(x,u)\, \mu_1(dx\times du) \rule[-15pt]{0pt}{15pt} \\
\mbox{Subject to} & \displaystyle \int_{E\times U} Af(x,u)\, \mu_0(dx\times du) + \int_{E\times U} Bf(x,u)\, \mu_1(dx\times du) = 0, \\
& \hfill \quad \forall\, f \in {\cal D}, \\
& \displaystyle \int_{E\times U} g_i(x,u)\, \mu_0(dx\times du) + \int_{E\times U} h_i(x,u)\, \mu_1(dx\times du) \leq K_i, \\
& \hfill \quad i=1,\ldots,m,\\
& \mu_0 \in {\cal P}(E\times U), \mu_1 \in {\cal M}(E\times U).
\end{array} 
\end{equation}
%Let $c^*$ denote the optimal value of the linear program (\ref{lta-lp}).  Suppose further that there exists a sequence $\{(\mu_0^n,\mu_1^n)\}$ of feasible solutions of (\ref{lta-lp}), with $\{\mu_1^n\}$ satisfying Condition~\ref{mu1-bounded}, such that $\lim_{n\rightarrow \infty} (\int c_0\, d\mu_0^n + \int c_1\, d\mu_1^n) = c^*$.  Then 
Moreover, there exists an optimal pair $(\mu_0^*,\mu_1^*)$.  Letting $\eta_0^*$ and $\eta_1^*$ be the transition functions defined by (\ref{etadefs}), an optimal absolutely continuous relaxed control is given in feedback form by $\{\Lambda^*_t=\eta_0^*(X^*(t),\cdot): t \geq 0\}$, where $X^*$ is the stationary process of Theorem~\ref{thm:stationary-existence} having one-dimensional distribution $\mu_1^{*E}$, and $\eta_1^*(x,\cdot)$ is an optimal relaxed singular control that is activated by the random measure $\Gamma^*$.
\end{thm}

\begin{proof}
First observe that if $J_0(\Lambda,\Gamma)$ is infinite for every relaxed solution of the singular, controlled martingale problem, then the minimal value of the linear program (\ref{lta-lp}) is also infinite.  For if not, then for some feasible pair $(\mu_0,\mu_1)$, Theorem~\ref{thm:stationary-existence} gives the existence of a stationary relaxed solution for which the long-term average cost is given by $\int c_0\, d\mu_0 + \int c_1\, d\mu_1$ resulting in a contradiction.  In this case, every solution is optimal (but not desired). 
 
Now let $(X,\Lambda,\Gamma)$ be a relaxed solution of the singular, controlled martingale problem for $(A,B)$ for which $J_0(\Lambda,\Gamma)<\infty$.  For $t > 0$, define measures $\mu_0^t$ and $\mu_1^t$ by
$$\begin{array}{rcll}
\mu_0^t(G) &=& \displaystyle t^{-1} \EE\left[\int_0^t \int_U I_{G}(X(s),u)\, \Lambda_s(du)\, ds\right], \rule[-15pt]{0pt}{15pt} & \quad G \in {\cal B}(E\times U), \\
\mu_1^t(G) &=& \displaystyle t^{-1} \EE\left[\int_{E\times U\times [0,t]} I_{G}(x,u)\, \Gamma(dx\times du \times ds)\right],& \quad G \in {\cal B}(E\times U).
\end{array}$$
Let $\{t_k:k\in \NN\}$ be any sequence of times with $t_k \rightarrow \infty$ as $k\rightarrow \infty$.  Using the same argument involving (\ref{lta-tightness}), the inf-compactness of $c_0$ and $c_1$ in Condition~\ref{costcnd1}(a) imply that $\{\mu_0^{t_k}: k\in \NN\}$ and $\{\mu_1^{t_k}: k\in\NN\}$ are relatively compact.  

Let $(\mu_0,\mu_1)$ be a limit point of $\{(\mu_0^{t_k},\mu_1^{t_k})\}$ for some subsequence $\{t_k\}$ having $t_k\rightarrow\infty$.  Consider $f \in {\cal D}$ and observe that since (\ref{mgp}) is a martingale, taking expectations and dividing by $t_k$ yields for each $k$, 
$$\EE\left[\mbox{$\frac{f(X(t_k))-f(X(0))}{t_k}$}\right] - \int Af(x,u)\, \mu_0^{t_k}(dx\times du) - \int Bf(x,u)\, \mu_1^{t_k}(dx\times du) = 0.$$
Letting $k \rightarrow \infty$, the first summand converges to $0$ since $f$ is bounded.  The convergence of the remaining terms to the adjoint relation of the main constraint of linear program (\ref{lta-lp}) follows using the same argument as in the proof of Proposition~\ref{feasible-limits}.  

Turning to an analysis of the cost and budget constraints, applications of the Skorohod representation theorem and Fatou's lemma again imply 
$$\int c_0\, d\mu_0 + \int c_1\, d\mu_1 \leq \liminf_{k_\ell \rightarrow \infty} \left(\int c_0\, d\mu_0^{t_{k_\ell}} + \int c_1\, d\mu_1^{t_{k_\ell}}\right) \leq J_0(\Lambda,\Gamma)$$
and similarly for each $i=1,\ldots,m$,
$$\int g_i\, d\mu_0 + \int h_i\, d\mu_1 \leq \liminf_{k_\ell \rightarrow \infty} \left(\int g_i\, d\mu_0^{t_{k_\ell}} + \int ch_i\, d\mu_1^{t_{k_\ell}}\right) \leq K_i.$$
Thus, the minimal cost of the linear program (\ref{lta-lp}) is a lower bound for the (\ref{ltacost}) over all relaxed solutions $(X,\Lambda,\Gamma)$ of the singular, controlled martingale problem for $(A,B)$ satisfying the budget constraints (\ref{lta-budget-constrs}).

Conversely, for every feasible pair of measures $(\mu_0,\mu_1)$ for which $\int c_0\, d\mu_0 + \int c_1\, d\mu_1 < \infty$, Theorem~\ref{thm:stationary-existence} establishes the existence of a stationary process $X$ and random measure $\Gamma$ having stationary increments such that, defining $\Lambda_t(\cdot) = \eta_0(X(t),\cdot)$ and $\tilde{\Gamma}(dx\times du\times ds) = \eta_1(x,du)\Gamma(dx\times ds)$, $(X,\Lambda,\tilde{\Gamma})$ is a relaxed solution of the singular, controlled martingale problem for $(A,B)$ and for which, for every $t>0$,  
\begin{eqnarray*}
\lefteqn{\int c_0\, d\mu_0 + \int c_1\, d\mu_1} \\
&=& t^{-1} \EE\left[\int_0^t \int_U c_0(X(s),u)\, \eta_0(X(s),du)\, ds + \int_{E\times [0,t]} \int_U c_1(x,u)\, \eta_1(x,du)\, \Gamma(dx \times ds)\right] 
\end{eqnarray*} 
and hence is the value of the expected long-term average cost $J_0(\Lambda,\Gamma)$ in (\ref{ltacost}).  Corollary~\ref{cor:lta-budget-constrs-meas} shows that $(X,\Lambda,\tilde\Gamma)$ also satisfies the budget constraints (\ref{lta-budget-constrs}).

It remains to show existence of an optimal pair $(\mu_0^*,\mu_1^*)$ when the optimal cost is finite.  Let $c^*$ denote the value of the linear program (\ref{lta-lp}) and let $\{(\mu_0^n,\mu_1^n): n \in \NN\}$ be a sequence for which
$$\lim_{n\rightarrow \infty} \left(\int c_0(x,u)\, \mu_0^n(dx\times du) + \int c_1(x,u)\, \mu_1(dx\times du)\right) = c^*.$$  
Then Proposition~\ref{feasible-limits} gives the existence of a feasible pair $(\mu_0^*,\mu_1^*)$ for which 
$$\int c_0(x,u)\, \mu_0^*(dx\times du) + \int c_1(x,u)\, \mu_1^*(dx\times du) \leq c^*,$$
establishing the optimality of $(\mu_0^*,\mu_1^*)$.
\end{proof}

\setcounter{equation}{0}

\section{Discounted control problem.} \label{secdisc}
We now turn to the reformulation of the singular control problem under the discounted criterion $J_\alpha(\Lambda,\Gamma)$ of (\ref{disccost}).  The first result gives the existence of solutions to the singular, controlled
martingale problem for $(A,B,\nu_0)$ having a desired cost.

\begin{thm} \label{exis:disc}
Let $A,B,\psi_A, \psi_B$ satisfy Condition \ref{gencnd}.
Suppose $\mu_0$ and $\mu_1$ satisfy (\ref{muadmis0}) and (\ref{integrability}) 
%\begin{eqnarray} \label{comp1}
%\int \psi_A(x,u) \,\mu_0(dx \times du) + \int \psi_B(x,u)\, \mu_1(dx \times du) < \infty
%\end{eqnarray}
and, for each  $f \in {\cal D}$, 
\begin{eqnarray} \label{id1}
\int \left[Af(x,u) + \alpha\left(\int f d\nu_0 - f(x)\right)\right] \mu_0(dx \times du) 
&+& \int Bf(x,u)\, \mu_1(dx \times du) = 0.\;\;\;\; ~
\end{eqnarray}
Let $\mu_i^E$ and $\eta_i$, $i = 0,1$, satisfy (\ref{etadefs}).  Then there exist a process $X$ and
a random measure $\Gamma$ on $E \times [0,\infty)$ such that (\ref{mg2}) is an
$\{\overline{\cal F}^X_{t+}\}$-martingale and 
\begin{eqnarray} \label{costrep} \nonumber
& &
\EE\left[\int_0^\infty \int_U e^{-\alpha s} c_0(X(s),u)\eta_0(X(s),du)\, ds 
+ \int_{E \times[0,\infty)} \int_U e^{-\alpha s} c_1(x,u)\eta_1(x,du)\, \Gamma(dx \times ds)\right] \\
& & \qquad \qquad = \alpha^{-1} \left[ \int c_0(x,u)\, \mu_0(dx \times du) + \int c_1(x,u)\, \mu_1(dx \times du)\right]
\end{eqnarray}
for all $c_0, c_1 \in B(E\times U)$ and for every nonnegative $c_0, c_1 \in M(E\times U)$.
\end{thm}

\begin{proof}
Enlarge the state space to $\{-1,+1\} \times E \times U$.  Let ${\cal D}_0 = \{ \phi f: \phi \in B\{-1,+1\}, f \in {\cal D}\}$ and define the generators $A^\alpha, B^\alpha: {\cal D}_0 \rightarrow C(\{-1,1\}\times E \times U)$ by
$$A^\alpha[\phi f](\theta,x,u) = \phi(\theta) Af(x,u) 
+ \alpha\left[\phi(-\theta) \int f\, d\nu_0 - \phi(\theta) f(x)\right]$$
and 
$$B^\alpha[\phi f](\theta,x,u) = \phi(\theta) Bf(x,u).$$
Observe that $A^\alpha$ and $B^\alpha$ satisfy Conditions~\ref{gencnd} with $\psi_A^\alpha(\theta,x,u) = \psi_A(x,u)$ and similarly, $\psi^\alpha_B(\theta,x,u) .= \psi_B(x,u)$.  Define measures $\tilde{\mu}_0$ and $\tilde{\mu}_1$ such that for each $h \in \bar{C}(\{-1,+1\} \times E \times U)$
%\begin{eqnarray*}
$$\int h(\theta,x,u)\, \tilde{\mu}_0(d\theta \times dx \times du) = \int \mbox{$\frac{1}{2}$} \left(h(-1,x,u) + h(+1,x,u)\right)\,
\mu_0(dx \times du) %\\
%& & + \int \frac{1}{2} \left(h(-1,x,u) + h(+1,x,u)\right)\,\mu_1(dx \times du) 
$$
%\end{eqnarray*}
and
$$\int h(\theta,x,u)\, \tilde{\mu}_1(d\theta \times dx\times du) 
= \int \mbox{$\frac{1}{2}$} \left(h(-1,x,u) + h(+1,x,u)\right)\, \mu_1(dx \times du).$$
Note that for $i=0,1$, 
$\tilde{\mu}_i(d\theta \times dx \times du) = \eta_i(x,du) \mu_i^E(dx) \cdot
\frac{1}{2}\left(\delta_{\{-1\}}(d\theta) + \delta_{\{+1\}}(d\theta)\right)$ 
so the transition functions $\tilde{\eta}_i(\theta,x,\cdot)$ used to disintegrate $\tilde{\mu}_i$ as in
(\ref{etadefs}) satisfy $\tilde{\eta}_i(\theta,x,\cdot) = \eta_i(x,\cdot)$ and thus only depend on
the value of $x$.

A straightforward calculation using (\ref{id1}) verifies that
\begin{eqnarray*}
\lefteqn{\int_{\{-1,+1\} \times E \times U} A^\alpha[\phi f](\theta,x,u)\, \tilde{\mu}_0(d\theta \times dx \times du)} \\
&\qquad +& \int_{\{-1,+1\} \times E \times U} B^\alpha[\phi f](\theta,x,u)\, \tilde{\mu}_1(d\theta \times dx \times du)
=0, \qquad \forall \phi f \in {\cal D}_0.
\end{eqnarray*}

By Theorem~\ref{thm:stationary-existence} there exist a $\{-1,+1\}\times E$-valued process 
$(\Theta,\tilde{X})$ and a random measure $\tilde{\Gamma}$
on $\{-1,+1\}\times E \times [0,\infty)$ such that
$(\Theta,\tilde{X})$ is stationary with 
$$\EE[I_{\{\theta\} \times H_1}(\Theta(t),\tilde{X}(t)) \eta_0(\tilde{X}(t),H_2)]
= \tilde{\mu}_0(\{\theta\} \times H_1 \times H_2), \quad \theta = \pm 1, H_1\in {\cal B}(E), H_2 \in {\cal B}(U)$$ for
each $t$, $\tilde{\Gamma}$ has stationary increments, $\tilde{\Gamma}(\{-1,+1\}\times E\times [0,t]) < \infty$
for each $t \geq 0$ and 
$\EE\left[\eta_1(x,du) \tilde{\Gamma}(d\theta \times dx \times [0,t])\right] 
= t \mu_1(dx \times du)\cdot \frac{1}{2}\left(\delta_{\{-1\}}(d\theta) + \delta_{\{+1\}}(d\theta)\right)$,
and for each $\phi f \in {\cal D}_0$,
\begin{eqnarray} \label{alphamg} \nonumber
\phi(\Theta(t))f(\tilde{X}(t)) &-& \phi(\Theta(0)) f(\tilde{X}(0))\\
&-& \int_0^t \int_U A^\alpha[\phi f](\Theta(s),\tilde{X}(s),u) \eta_0(\tilde{X}(s),du)\, ds 
\\ \nonumber
&-& \int_{\{-1,+1\}\times E\times [0,t]} \int_U B^\alpha[\phi f](\theta,x,u) \eta_1(x,du) \tilde{\Gamma}(d\theta \times dx \times ds)
\end{eqnarray}
is an $\{\bar{\cal F}_{t+}^{\;\Theta,\tilde{X}}\}$-martingale.

Taking $f \equiv 1$ in (\ref{alphamg}), we have
$$\phi(\Theta(t)) - \phi(\Theta(0)) - \int_0^t \alpha\left(\phi(-\Theta(s)) - \phi(\Theta(s))\right)\, ds$$
is a martingale and hence $\Theta$ is a (stationary) continuous time Markov chain which jumps between states $\{-1\}$ and
$\{+1\}$ at rate $\alpha$.  
Taking $\phi \equiv 1$ in (\ref{alphamg}), it follows that 
\begin{eqnarray} \label{statmg} \nonumber
f(\tilde{X}(t)) - f(\tilde{X}(0)) &-& \int_0^t \int_U \left[Af(\tilde{X}(s),u) + \alpha\left(\int f\, d\nu_0 - f(\tilde{X}(s))\right)\right]
\eta_0(\tilde{X}(s),du)\, ds \\ 
&-& \int_{E \times [0,t]} \int_U Bf(x,u) \eta_1(x,du) \tilde{\Gamma}(dx \times ds)
\end{eqnarray}
is an $\{\bar{\cal F}_{t+}^{\;\Theta,\tilde{X}}\}$-martingale.  Note we slightly abuse notation by using $\tilde{\Gamma}(\{-1,+1\}\times dx\times ds) = \tilde{\Gamma}(dx \times ds)$.

Now let $\tau_0 = \sup\{t<0: \Theta(t) \neq \Theta(0)\}$, let $\tau_1 = \inf\{t\geq 0: \Theta(t) \neq \Theta(0)\}$
and for $k \geq 1$, let $\tau_{k+1} = \inf\{t > \tau_k: \Theta(t) \neq \Theta(\tau_k)\}$.  Note that $\{\tau_k: k\geq 1\}$
give the jump times of the Markov chain $\Theta$ and thus have exponentially distributed interarrival times.  The
collection
$\{\tau_k: k \geq 1\}$ are renewal times of $\Theta$ (though they may not be renewal times of $(\Theta,\tilde{X})$ in that
the cycles may not be independent and identically distributed).  

For $t \geq 0$ define $X(t) = \tilde{X}(\tau_1+t)$, 
$\Gamma(\{\theta\} \times H \times [0,t]) = \tilde{\Gamma}(\{\theta\} \times H \times [\tau_1,\tau_1+t])$, where
$\theta=\pm 1$, $H \in {\cal B}(E)$, and ${\cal F}_t = \bar{\cal F}_{(\tau_1+t)+}^{\;\Theta,\tilde{X}}$.
The optional sampling theorem implies 
\begin{eqnarray*}
f(X(t)) &-& f(X(0)) \\
&-& \int_0^t \int_U \left[Af(X(s),u) + \alpha \left(\int f\,d\nu_0 - f(X(s))\right)\right]\eta_0(X(s),du)\, ds\\
&-& \int_{E \times [0,t]} \int_U Bf(x,u) \eta_1(x,du) \Gamma(dx \times ds) \\
\end{eqnarray*}
\begin{eqnarray*}
\rule{18pt}{0pt} &=& f(\tilde{X}(\tau_1+t)) - f(\tilde{X}(\tau_1)) \\
& & - \int_{\tau_1}^{\tau_1+t} \int_U\left[Af(\tilde{X}(s),u) 
+ \alpha\left(\int f\, d\nu_0 - f(\tilde{X}(s))\right)\right]\eta_0(\tilde{X}(s),du) ds\\
& & - \int_{E\times [\tau_1,\tau_1+t]}\int_U Bf(x,u) \eta_1(x,du) \tilde{\Gamma}(dx \times du \times ds) 
\end{eqnarray*}
is a martingale with respect to $\{{\cal F}_t\}=\{\bar{\cal F}_{(\tau_1+t)+}^{\; \Theta,\tilde{X}}\}$.  

Now for $t \geq 0$, define
$$L(t) = [\alpha(\tau_1-\tau_0)]^{-1} e^{\alpha t} I_{[0,\tau_2-\tau_1)}(t)$$
and observe that $L$ is a mean 1, $\{{\cal F}_t\}$-martingale.  
Define a new probability measure $\hat{P}$ having Radon-Nikodym derivative $L(t)$ on $\{{\cal F}_t\}$
with respect to the original probability $P$.  It follows that, for each $f \in \doma$,
\begin{eqnarray} \label{mg3}
L(t) f(X(t)) &-& \int_0^t \int_U L(s) Af(X(s),u)\, \eta_0(X(s),du)\, ds \\ \nonumber
&-& \int_{E \times [0,t]}\int_U L(s) Bf(x,u) \eta_1(x,du)\, \Gamma(dx \times ds)
\end{eqnarray}
is also an $\{{\cal F}_t\}$-martingale under $P$ and thus for each $n \geq 1$, $0 \leq t_1 \leq \cdots \leq t_n < t_{n+1}$,
$f \in \doma$
and $h_1,\ldots, h_n \in \bar{C}(E)$,
\begin{eqnarray*}
0 &=& \EE\left[\prod_{i=1}^n h_i(X(t_i)) \left(\rule{0pt}{18pt} L(t_{n+1})f(X(t_{n+1})) - L(t_n) f(X(t_n))\right. \right. \\
& & \qquad - \int_{t_n}^{t_{n+1}} \int_U L(s) Af(X(s),u) \eta_0(X(s),du) ds \\
& & \qquad - \left. \left. \int_{E \times [t_n,t_{n+1}]} \int_U L(s)Bf(x,u) \eta_1(x,du) \Gamma(dx \times ds) \right) \right]
 \\
&=& \EE\left[ \prod_{i=1}^n h_i(X(t_i)) [\alpha(\tau_1-\tau_0)]^{-1} \left(\rule{0pt}{18pt} 
e^{\alpha t_{n+1}}I_{[0,\tau_2-\tau_1)}(t_{n+1}) f(X(t_{n+1}) 
\right. \right. \\
& & \qquad -  \EE[e^{\alpha t_{n+1}}I_{[0,\tau_2-\tau_1)}(t_{n+1}) | {\cal F}_{t_n}] f(X(t_n)) \\
& & \qquad - \int_{t_n}^{t_{n+1}} \int_U \EE[e^{\alpha t_{n+1}}I_{[0,\tau_2-\tau_1)}(t_{n+1}) | {\cal F}_{s}]
Af(X(s),u)\eta_0(X(s),du) ds \\
& & \qquad - \left. \left. \int_{E \times [t_n,t_{n+1}]} \int_U 
\EE[e^{\alpha t_{n+1}}I_{[0,\tau_2-\tau_1)}(t_{n+1}) | {\cal F}_{s}]
Bf(x,u) \eta_1(x,du) \Gamma(dx \times ds) \rule{0pt}{18pt} \right)\right] \\
&=& \EE^{\hat{P}}\left[\prod_{i=1}^n h_i(X(t_i)) \left(\rule{0pt}{18pt} f(X(t_{n+1})) - f(X(t_n)) \right. \right. 
- \int_{t_n}^{t_{n+1}} \int_U Af(X(s),u)\eta_0(X(s),du) ds \\
& & \qquad \left. \left.
- \int_{E \times [t_n,t_{n+1}]} \int_U  Bf(x,u) \eta_1(x,du) \Gamma(dx \times ds) \rule{0pt}{18pt} \right)\right] .
\end{eqnarray*}
Thus letting $\Lambda_s(du) = \eta_0(X(s),du)$ and $\hat{\Gamma}(dx \times du \times ds) = \eta_1(x,du)\Gamma(\{-1,+1\}\times dx \times ds)$,
the triplet $(X,\Lambda,\hat{\Gamma})$ is a solution of the singular, controlled martingale problem for $(A,B,\nu_0)$
under $\hat{P}$.

We now show that this solution satisfies (\ref{costrep}).  Define the random measure $\tilde{\Phi}$ on $E \times [0,\infty)$
such that for every bounded, continuous $h$
$$\int_{E \times [0,\infty)} h(x,s)\, \tilde{\Phi}(dx \times ds) = \int_0^\infty h(\tilde{X}(s),s)\, ds.$$
Note that $\tilde{\Phi}$ has stationary increments and $\EE[\tilde{\Phi}(\cdot \times [0,t])] = t \mu_0^E(\cdot)$ for every $t > 0$.
Also recall $\tilde{\Gamma}$ has stationary increments and $\EE[\tilde{\Gamma}(\cdot \times [0,t])] = t \mu_1^E(\cdot)$.
The following argument applies to both random measures $\tilde{\Phi}$ and $\tilde{\Gamma}$ so let $\tilde{\Psi}$ denote either random measure, 
let $\Psi(\cdot \times[0,t]) = \tilde{\Psi}(\cdot \times [\tau_1,\tau_1+t])$ and let $\eta$ denote the
appropriate choice of $\eta_0$ or $\eta_1$.

For each $h \in \bar{C}(E \times U)$,
\begin{eqnarray*}
\lefteqn{\EE^{\hat{P}}\left[\alpha \int_{E \times [0,T]}\int_U 
e^{-\alpha t} h(x,u)\eta(x,du)\,\Psi(dx \times dt)\right] } \\
&=& \EE\left[\alpha [\alpha(\tau_1-\tau_0)]^{-1} e^{\alpha T}I_{[0,\tau_2-\tau_1)}(T) \int_{E \times [0,T]}\int_U 
e^{-\alpha t} h(x,u)\eta(x,du)\,\Psi(dx \times dt)\right]  \\
\end{eqnarray*}
\begin{eqnarray*}
&=& \EE\left[(\tau_1 - \tau_0)^{-1} \int_{E \times [0,T]}\int_U \EE\left[e^{\alpha T}I_{[0,\tau_2-\tau_1)}(T)|{\cal F}_t\right]
e^{-\alpha t} h(x,u)\eta(x,du)\,\Psi(dx \times dt)\right]  \\
&=& \EE\left[(\tau_1 - \tau_0)^{-1} \int_{E \times [\tau_1,\tau_2\wedge T]}\int_U 
h(x,u)\eta(x,du)\,\tilde{\Psi}(dx \times dt)\right]  
\end{eqnarray*}
and letting $T\rightarrow \infty$ yields
\begin{eqnarray} \label{alphaid} \nonumber
\lefteqn{\EE^{\hat{P}}\left[\alpha \int_{E \times [0,\infty)}\int_U 
e^{-\alpha t} h(x,u)\eta(x,du)\,\Psi(dx \times dt)\right] } \\
&=& \EE\left[(\tau_1 - \tau_0)^{-1} \int_{E \times [\tau_1,\tau_2]}\int_U 
h(x,u)\eta(x,du)\,\tilde{\Psi}(dx \times dt)\right].
\end{eqnarray}
Applying Lemma~\ref{statmeas1} to the right-hand-side when $\tilde{\Psi} = \tilde{\Phi}$ yields 
$$\EE^{\hat{P}}\left[\alpha \int_0^\infty\int_U e^{-\alpha t} h(X(t),u)\eta_0(X(t),du)\,dt\right] 
= \int h(x,u) \mu_0(dx\times du)$$
and when $\tilde{\Psi} = \tilde{\Gamma}$ we have
$$\EE^{\hat{P}}\left[\alpha \int_{E \times [0,\infty)}\int_U e^{-\alpha t} h(x,u)\eta_1(x,du)\,\Gamma(dx \times dt)\right] 
= \int h(x,u) \mu_1(dx\times du)$$
establishing the result.
\end{proof}
\medskip

As with Theorem~\ref{thm:stationary-existence} for the long-term average criterion, Theorem~\ref{exis:disc} gives the existence of a solution of the singular, controlled martingale problem whose discounted occupation measures are the pair $(\mu_0,\mu_1)$ satisfying (\ref{id1}).  This result also extends to the budget constraints.

\begin{coro}
Let $A$, $B$, $\psi_A$, $\psi_B$, $\mu_0$ and $\mu_1$ satisfy the hypotheses of Theorem~\ref{exis:disc} and let $X$ be a process and $\Gamma$ a random measure resulting from the theorem.  Define the relaxed control $\Lambda$ so that $\Lambda_s(\cdot) = \eta_0(X(s),\cdot)$ for $s \geq 0$ and the random measure $\tilde\Gamma$ by 
$$\tilde\Gamma(G_1\times G_2 \times [0,t]) = \int_{G_1 \times [0,t]} \eta_1(x,G_2)\, \Gamma(dx\times ds), \quad G_1 \in {\cal B}(E), G_2 \in {\cal B}(U), t\geq 0$$
so that $(X,\Lambda,\tilde\Gamma)$ is a solution of the singular, controlled martingale problem for $(A,B,\nu_0)$.  Then $(X,\Lambda,\tilde\Gamma)$ satisfies the discounted budget constraints (\ref{disc-budget-constrs}) if and only if
$$\int_{E\times U} g_i(x,u)\, \mu_0(dx\times du) + \int_{E\times U} h_i(x,u)\, \mu_1(dx\times du) \leq \alpha K_i, \quad i=1,\ldots,m.$$
\end{coro}

\begin{proof}
Since each $g_i$ and $h_i$ in the budget constraints (\ref{lta-budget-constrs}) is bounded below, the result follows immediately by writing $g_i = g_i^+ - g_i^-$ and $h_i^+ - h_i^-$ and applying the (\ref{costrep}).
\end{proof}

Now consider the problem of minimizing the discounted cost (\ref{disccost}) over solutions $(X,\Lambda,\Gamma)$ of the singular, controlled martingale problem for $(A,B)$ satisfying the discounted budget constraints (\ref{disc-budget-constrs}). 

\begin{thm} \label{thm:disc-lp}
Assume Conditions~\ref{gencnd} and (\ref{costcnd1}) hold and let $\nu_0 \in {\cal P}(E)$ be given.  Let $\alpha > 0$ and define the generator $A^\alpha$ by $A^\alpha f = Af + \alpha (\int f\, d\nu_0 - f)$.  Then the problem of minimizing the discounted cost (\ref{disccost}) over admissible solutions $(X,\Lambda,\Gamma)$ of the discounted singular, controlled martingale problem for $(A,B,\nu_0)$ that satisfy the budget constraints (\ref{disc-budget-constrs}) is equivalent to the linear program 
\begin{equation} \label{disc-lp}
\begin{array}{ll}
\mbox{Minimize} & \displaystyle \alpha^{-1}\left(\int_{E\times U} c_0(x,u)\, \mu_0(dx\times du) + \int_{E\times U} c_1(x,u)\, \mu_1(dx\times du)\right) \rule[-15pt]{0pt}{15pt} \\
\mbox{Subject to} & \displaystyle \int_{E\times U} A^\alpha f(x,u)\, \mu_0(dx\times du) + \int_{E\times U} Bf(x,u)\, \mu_1(dx\times du) = 0, \\
& \hfill \forall\, f \in {\cal D}, \\
& \displaystyle \int_{E\times U} g_i(x,u)\, \mu_0(dx\times du) + \int_{E\times U} h_i(x,u)\, \mu_1(dx\times du) \leq \alpha K_i, \\
& \hfill i=1,\ldots,m,\\
& \mu_0 \in {\cal P}(E\times U), \mu_1 \in {\cal M}(E\times U).
\end{array} 
\end{equation}
Moreover, there exists an optimizing pair $(\mu_0^*,\mu_1^*)$ and, letting $\eta_0^*$ and $\eta_1^*$ be the transition functions defined by (\ref{etadefs}), an optimal absolutely continuous relaxed control is given in feedback form by $\{\Lambda^*_t=\eta_0^*(X^*(t),\cdot): t \geq 0\}$, in which $(X^*,\Gamma^*)$ is the process of Theorem~\ref{exis:disc}, and $\eta_1^*(x,\cdot)$ is an optimal relaxed singular control that is activated by the random measure $\Gamma^*$.
\end{thm}

\begin{proof}
Similar to the argument in Theorem~\ref{thm:lta-lp}, if every solution $(X,\Lambda,\Gamma)$ of the singular, controlled martingale problem for $(A,B,\nu_0)$ has infinite value for the discounted cost (\ref{disccost}), then the value of each feasible solution of the linear program (\ref{disc-lp}) is also infinite and every solution is optimal, though not desired.

Let $(X,\Lambda,\Gamma)$ be an admissible solution of the singular, controlled martingale problem for $(A,B,\nu_0)$ for which (\ref{disccost}) is finite and the budget constraints (\ref{disc-budget-constrs}) are satisfied.  Define the measures $\mu_0$ and $\mu_1$ by 
$$\begin{array}{rcll}
\mu_0(G) &=& \displaystyle \alpha \EE\left[\int_0^\infty \int_U e^{-\alpha s} I_{G}(X(s),u)\, \Lambda_s(du)\, ds\right] \rule[-15pt]{0pt}{15pt}, & G\in {\cal B}(E\times U), \\
\mu_1(G) &=& \displaystyle \alpha \EE\left[\int_{E\times U\times [0,\infty)}e^{-\alpha s} I_{G}(x,u)\, \Gamma(dx\times du \times ds)\right], & G\in {\cal B}(E\times U)
\end{array}$$
and note that $\mu_1 \in {\cal M}(E\times U)$ by Condition~\ref{costcnd1}(c).  It immediately follows that the value of (\ref{disccost}) is given by the objective function of (\ref{disc-lp}) and the budget constraints are represented by the collection of linear programming constraints involving $g_i$ and $h_i$ for $i=1,\ldots, m$.  The first set of linear programming constraints follow by an application of It\^{o}'s formula on $e^{-\alpha t} f(X(t))$, taking expectations and passing to the limit as $t\rightarrow \infty$.

Now consider any feasible pair $(\mu_0,\mu_1)$ of (\ref{disc-lp}) for which the value of the objective function is finite.  Condition~\ref{costcnd1} implies that (\ref{integrability}) and (\ref{id1}) hold.  By Theorem~\ref{exis:disc}, there exists a process $X$ and a random measure $\Gamma$ for which the process $X$, the feedback control $\{\eta_0(X(t),\cdot):t \geq 0\}$ and the random measure $\tilde{\Gamma}(dx\times du\times ds) = \eta_1(x,du) \Gamma(dx\times ds)$ is a relaxed solution of the singular, controlled martingale problem for $(A,B,\nu_0)$.  Moreover, the corresponding discounted cost (\ref{disccost}) is given by the value of the objective function of (\ref{disc-lp}) and the budget constraints (\ref{disc-budget-constrs}) also follow from (\ref{costrep}).  

Thus to each admissible relaxed solution $(X,\Lambda,\Gamma)$ of the singular, controlled martingale problem for $(A,B,\nu_0)$, there exists a feasible pair of measures $(\mu_0,\mu_1)$ to (\ref{disc-lp}) and, to each feasible pair $(\mu_0,\mu_1)$, there exists a relaxed solution of the singular, controlled martingale problem for $(A,B,\nu_0)$ for which, in both cases, the value $J_\alpha(\Lambda,\Gamma;\nu_0)$ of (\ref{disccost}) and the objective function are equal.  

Let $c^*$ denote the infimal value of (\ref{disc-lp}).  Existence of an optimizing pair $(\mu_0^*,\mu_1^*)$ follows by applying Proposition~\ref{feasible-limits} to any sequence $\{(\mu_0^n,\mu_1^n)\}$ for which $\int c_0\, d\mu_0^n + \int c_1\, d\mu_1^n < c_* + \frac{1}{n}$.  Optimality of the relaxed feedback controls $\eta_0$ and $\eta_1$ follows from Theorem~\ref{exis:disc}.
\end{proof}

\begin{rem} \label{rescaling}
We observe that a simple rescaling argument shows the equivalence of the discounted control problem with the linear program
\begin{equation} \label{disc-lp2}
\begin{array}{lll}
\mbox{Min.} & \displaystyle \int_{E\times U} c_0\, d\mu_0 + \int_{E\times U} c_1\, d\mu_1 & \rule[-15pt]{0pt}{15pt} \\
\mbox{Subj. to} & \displaystyle \int_{E\times U} (Af-\alpha f)\, d\mu_0 + \int_{E\times U} Bf\, d\mu_1 = - \int f\, d\mu_0, & \quad \forall\, f \in {\cal D}, \\
& \displaystyle \int_{E\times U} g_i\, d\mu_0 + \int_{E\times U} h_i\, d\mu_1 \leq K_i, & \quad i=1,\ldots, m, \\
& \mu_0 \in {\cal M}(E\times U), \mu_1 \in {\cal M}(E\times U). &
\end{array} 
\end{equation}
The mass condition $\mu_0(E\times U) = \frac{1}{\alpha}$ follows from the first family of constraints by considering $f\equiv 1$.
\end{rem}

\setcounter{equation}{0}

\section{Existence of Optimal Strict Feedback Controls} \label{sect:strict}
We now provide a set of sufficient conditions under which one is able to strengthen the existence of optimal controls from the class of relaxed controls to the class of strict controls.  %A $U$-valued process $u$ is a strict control for the singular, controlled martingale problem for $(A,B,\nu_0)$ if there exists some filtration $\{{\cal F}_t\}$, $E$-valued process $X$ and random measure $\Gamma$ on $E\times U\times [0,\infty)$ such that $(X,u,\Gamma_t)$ is $\{{\cal F}_t\}$-progressively measurable, $X(0)$ has distribution $\nu_0$ and, for every $f \in {\cal D}$,
%$$f(X(t)) - f(X(0)) - \int_0^t Af(X(s),u(s))\, ds - \int_{E\times U \times [0,t]} Bf(x,u)\, \Gamma(dx \times du \times ds)$$
%is an $\{{\cal F}_t\}$-martingale.  
A strict feedback control is a measurable function $u:E\rightarrow U$ for which there exists some filtration $\{{\cal F}_t\}$, $E$-valued process $X$ and random measure $\Gamma$ on $E\times [0,\infty)$ such that $(X,\Gamma_t)$ is $\{{\cal F}_t\}$-progressively measurable, $X(0)$ has distribution $\nu_0$ and, for every $f \in {\cal D}$,
$$f(X(t)) - f(X(0)) - \int_0^t Af(X(s),u(X(s)))\, ds - \int_{E \times [0,t]} Bf(x,u(x))\, \Gamma(dx \times ds)$$
is an $\{{\cal F}_t\}$-martingale.  We establish conditions for the existence of an optimal strict feedback control.

This section extends the results in \cite{dufo:12} from absolutely continuously controlled processes to include singular and singularly controlled processes.  Lemma~3.1 of \cite{dufo:12}, which is a mild extension of Theorem~A.9 of \cite{haus:90}, provides a crucial measurable selection result.  The proof of Theorem~\ref{thm:strict-control} below also relies on a result in \cite{warg:72}.  For completeness of exposition, we state these results.

\begin{thm}[Theorem~I.6.13 (p.~145) of \cite{warg:72}] \label{warga}
Let $(S,\Sigma,\mu)$ be a probability measure space, ${\cal X}$ a separable Banach space, $C$ a closed, convex subset of ${\cal X}$, and $f:S\rightarrow C$ $\mu$-integrable.  Then
$$\int f(s)\, \mu(ds) \in C.$$
\end{thm}

\begin{lem}[Lemma~3.1 of \cite{dufo:12}] \label{lem-dufo}
Let $\psi : E\times U \mapsto \overline{\RR}^{n+1}$ and $\phi:E\times U \rightarrow \RR^{\NN}$ be measurable functions with $\psi$ bounded below, $\psi(x,\cdot)$ lower semi-continuous and $\phi(x,\cdot)$ continuous for each $x$. Assume also that $\psi_{j_0}$ is inf-compact for some $0\leq j_0 \leq n$.  First define the set $k(x) = \{(z,u)\in \RR^{n+1}\times U: z_i \geq \psi_i(x,u), i=0,\ldots,n\}$ and then define
$$K(x)=\big\{ (z,\phi(x,u))\in \RR^{n+1}\times \RR^{\NN}: (z,u) \in k(x) \big\}.$$
Let $h_{1}: E \mapsto \RR^{n+1}$ and $h_{2} : E \mapsto \RR^{\NN}$ be measurable functions such that for all $x\in E$, $(h_{1}(x),h_{2}(x))\in K(x)$. Then there exists a measurable function $\widehat{u} : E \mapsto U$ such that
for all $x\in E$, $h_{1}(x) \geq \psi(x,\widehat{u}(x))$ and $h_{2}(x) = \phi(x,\widehat{u}(x))$.
\end{lem}
\medskip

\noindent
{\em Long-term Average Problems.}\/  Consider the long-term average control problem of minimizing $J_0(\Lambda,\Gamma)$ in (\ref{ltacost}) over solutions $(X,\Lambda,\Gamma)$ of the singular, controlled martingale problem satisfying the budget and resource constraints (\ref{lta-budget-constrs}).  We introduce the key convexity condition under which the existence of an optimal strict control can be selected.  

Recall, for some $m<\infty$, the functions in the budget constraints $g_i,h_i: E\times U \rightarrow \RR^+$, $i = 1,\ldots, m$, are lower semicontinuous and bounded below.  Let $\{f_k\}\subset {\cal D}$ denote the countable collection of Condition~\ref{gencnd}(iii).  For each $x\in E$, define the sets 
\begin{equation} \label{kappa-def}
\begin{array}{rcl}
\kappa(x) &=& \{(z,u)\in \RR^{2m+2}\times U: z_0 \geq c_0(x,u), z_1 \geq c_1(x,u), \\
& & \qquad \qquad \qquad \qquad \qquad z_{2i} \geq g_i(x,u), z_{2i+1} \geq h_i(x,u),  i = 1, \ldots, m\}
\end{array}
\end{equation}
and 
\begin{equation} \label{K-def}
{\cal K}(x) = \left\{\big(z,(Af_k(x,u),Bf_k(x,u))_{k\in \NN}\big)\in \RR^{2m+2}\times\RR^{\NN}: (z,u) \in \kappa(x)\right\}.
\end{equation}
Observe that in the set ${\cal K}(x)$, the $z_0$-coordinate gives the epi-graph of $c_0(x,\cdot)$, the $z_1$-coordinate is the epi-graph of $c_1(x,\cdot)$ and similarly, for $i\in \{1,\ldots,m\}$, $z_{2i+1}$ produces the epi-graph of $g_i(x,\cdot)$ while $z_{2i+2}$ yields the epi-graph of $h_i(x,\cdot)$.  

\begin{cnd} \label{convexity}
For each $x$, the set ${\cal K}(x)$ is closed and convex.  
\end{cnd}

We only consider control problems for which the cost associated with some control (and hence the optimal value) is finite since otherwise every control policy is trivially optimal and there is nothing to prove.

\begin{thm} \label{thm:strict-control}
Assume Conditions~\ref{gencnd}, \ref{costcnd1} and \ref{convexity} hold.  Let $c^*<\infty$ denote the optimal value of the linear program (\ref{lta-lp}).  Then there exists an optimal strict control $u^*:E\rightarrow U$.
\end{thm}

\begin{proof}
Under the hypotheses of the theorem, Theorem~\ref{thm:lta-lp} establishes the existence of an optimal pair $(\mu_0^*,\mu_1^*)$ for the linear program (\ref{lta-lp}).  For $i=0,1$, let $\eta_i^*$ be the transition function satisfying (\ref{etadefs}) for the measures $(\mu_0^*,\mu_1^*)$.  Apply Theorem~\ref{thm:stationary-existence} to obtain $(X^*,\Gamma^*)$ and define the relaxed control $\Lambda^*$ by $\Lambda_s^*(\cdot) = \eta_0^*(X^*(s),\cdot)$ for $s \geq 0$ and the $E\times U \times [0,\infty)$-measure valued random variable $\tilde{\Gamma}^*$ by 
$$\tilde{\Gamma}^*(G_1\times G_2 \times [0,t]) = \int_{G_1\times [0,t]} \eta_1^*(x,G_2)\, \Gamma^*(dx\times ds).$$
It then follows that $(X^*,\Lambda^*,\tilde\Gamma^*)$ is a stationary solution of the singular, controlled martingale problem for $(A,B)$ satisfying the budget constraints (\ref{lta-budget-constrs}) such that
$$J_0(\Lambda^*,\tilde\Gamma^*) = \int c_0(x,u)\, \mu_0^*(dx\times du) + \int c_1(x,u)\, \mu_1^*(dx\times du).$$

Using $\eta_0^*$ and $\eta_1^*$, define
$$\begin{array}{rclcrcll}
\overline{c}_0(x) &=&\displaystyle \int_U c_0(x,u)\, \eta_0^*(x,du), &\qquad& \overline{c}_1(x)&=& \displaystyle \int_U c_1(x,u)\, \eta_1^*(x,du),  & \rule[-12pt]{0pt}{12pt} \\
\overline{g}_i(x) &=&\displaystyle \int_U g_i(x,u)\, \eta_0^*(x,du), &\qquad& \overline{h}_i(x)&=& \displaystyle \int_U h_i(x,u)\, \eta_1^*(x,du),  & \quad i=1\ldots,m, \rule[-12pt]{0pt}{12pt} \\
\overline{Af}_k(x) &=& \displaystyle \int_U Af_k(x,u)\, \eta_0^*(x,du) & & \overline{Bf}_k(x)&=& \displaystyle \int_U Bf_k(x,u)\, \eta_1^*(x,du), & \quad k\in \NN.
\end{array}$$
Then Theorem~\ref{warga} implies that for each $x$, 
$$(\bar{c}_0(x),\bar{c}_1(x),\bar{g}_1(x),\bar{h}_1(x),\ldots,\bar{g}_m(x),\bar{h}_m(x), \bar{Af}_1(x),\bar{Bf}_1(x), \bar{Af}_2(x),\bar{Bf}_2(x),\ldots) \in {\cal K}(x).$$
Consequently, Lemma~\ref{lem-dufo} establishes the existence of measurable functions $u^*:E\rightarrow U$ and $v^*:E \rightarrow \RR_+^{2m+2}$ such that for all $x \in E$,
\begin{eqnarray} \label{u-star-id1a} 
c_0(x,u^{*}(x)) + v_0^*(x) = \bar{c}_0(x)  &=& \int_{U} c_0(x,u)\, \eta_0^{*}(x,du)         \rule[-15pt]{0pt}{12pt} \\ \label{u-star-id1b} 
c_1(x,u^{*}(x)) + v_0^*(x) = \bar{c}_1(x)  &=& \int_{U} c_1(x,u)\, \eta_1^{*}(x,du)         \rule[-15pt]{0pt}{12pt} \\ \label{u-star-id2a}
Af_{k}(x,u^{*}(x))     = \bar{Af}_k(x) &=& \int_{U} Af_{k}(x,u)\, \eta_0^{*}(x,du),  \quad k\in \NN, \rule[-15pt]{0pt}{12pt}\\ \label{u-star-id2b}
Bf_{k}(x,u^{*}(x))     = \bar{Af}_k(x) &=& \int_{U} Bf_{k}(x,u)\, \eta_1^{*}(x,du),  \quad k\in \NN, \rule[-15pt]{0pt}{12pt}\\ \label{u-star-id3a}
g_i(x,u^*(x)) + v_i^*(x) = \bar{g}_i(x)  &=& \int_U g_i(x,u)\, \eta_0^*(x,du),         \qquad i\in \{1,\ldots,m\}\rule[-15pt]{0pt}{12pt}\\ \label{u-star-id3b}
h_i(x,u^*(x)) + v_i^*(x) = \bar{h}_i(x)  &=& \int_U h_i(x,u)\, \eta_1^*(x,du),         \qquad i\in \{1,\ldots,m\}.
\end{eqnarray}
Since $(X^*,\Lambda^*,\tilde\Gamma^*)$ is a solution of the singular, controlled martingale problem for $(A,B)$, using identities (\ref{u-star-id2a}) and (\ref{u-star-id2b}) in (\ref{mgp}) implies that for each $k$,
\begin{eqnarray*}
f_k(X^*(t)) - f_k(X^*(0)) &-& \int_0^t Af_k(X^*(s),u^*(X^*(s)))\, ds \\
&+& \int_{E\times [0,t]} Bf_k(x,u^*(x))\, \Gamma^*(dx\times ds)
\end{eqnarray*}
is a martingale.  Condition~\ref{gencnd}(iii) therefore implies that $(X^*,\delta_{\{u^*(X^*)\}},\delta_{\{u^*(\cdot)\}} \Gamma^*)$ is a solution of the singular, controlled martingale problem for $(A,B)$ and thus $u^*$ is a strict control.  Identities (\ref{u-star-id2a}) and (\ref{u-star-id2b}) also imply that 
$$\int Af(x,u^*(x))\, \mu_0^{*E}(dx) + \int Bf(x,u^*(x))\, \mu_1^{*E}(dx) = 0 \qquad \forall f\in {\cal D}.$$
In addition, identities (\ref{u-star-id3a}) and (\ref{u-star-id3b}) yield for $i=1,\ldots, m$,
\begin{eqnarray*}
\lefteqn{\int g_i(x,u^*(x))\, \mu_0^{*E}(dx) + \int h_i(x,u^*(x))\, \mu_1^{*E}(dx)} \\
&\qquad \leq& \int g_i(x,u)\, \mu_0^*(dx\times du) + \int h_i(x,u)\, \mu_1^*(dx\times du) \leq K_i.
\end{eqnarray*} 
Thus $(\delta_{\{u^*(\cdot)\}}\mu_0^{*E},\delta_{\{u^*(\cdot)\}}\mu_1^{*E})$ is feasible for the linear program (\ref{lta-lp}).   Optimality of $(\mu_0^*,\mu_1^*)$ along with the identities (\ref{u-star-id1a}) and (\ref{u-star-id1b}) now establishes that
\begin{eqnarray*}
\lefteqn{\int c_0(x,u^*(x))\, \mu_0^{*E}(dx) + \int c_1(x,u^*(x))\, \mu_1^{*E}(dx)} \\
&\quad =& \int c_0(x,u)\, \mu_0^*(dx\times du) + \int c_1(x,u)\, \mu_1^*(dx\times du)
\end{eqnarray*}
and hence $u^*$ is an optimal strict feedback control.
\end{proof}
\medskip

\noindent
{\em Discounted Problems.}\/  In comparing the linear programs in Theorem~\ref{thm:lta-lp} and Theorem~\ref{thm:disc-lp}, the only difference lies in the generators $A$ and $A^\alpha$.  Intuitively, this observation means that one may solve a discounted control problem for a process having generator $A$ by solving a long-term average control problem for a process with generator $A^\alpha$.  Under $A^\alpha$, the dynamics follows the evolution associated with generator $A$ for an exponentially distributed length of time.  At the occurrence of this exponential time, the process reinitializes and proceeds to continually cycle.  Unfortunately, these cycles may not be independent so the proof of Theorem~\ref{exis:disc} cannot rely on renewal arguments. 

This relation between $A^\alpha$ and $A$ is important with regard to the existence of an optimal strict control for discounted control problems.  For each $x\in E$, define $\kappa(x)$ as in (\ref{kappa-def}) and then define the set ${\cal K}_\alpha(x)$ by 
$${\cal K}_\alpha(x) = \left\{\big(z,(A^\alpha f_k(x,u),Bf_k(x,u))_{k\in \NN}\big)\in \RR^{2m+2}\times\RR^{\NN}: (z,u) \in \kappa(x)\right\}.$$
Under the condition that ${\cal K}_\alpha$ be closed and convex, Theorem~\ref{thm:strict-control} then establishes existence of the desired optimal strict control.  

Notice that $A^\alpha f(x,u) = Af(x,u) - \alpha f(x)$ so $A^\alpha f(x,u)$ is a translation of $Af(x,u)$ by $-\alpha f(x)$ for each $x \in E$.  Thus the set ${\cal K}_\alpha(x)$ is merely a translation of the set ${\cal K}(x)$ and hence ${\cal K}_\alpha(x)$ is closed and convex if and only if ${\cal K}(x)$ is closed and convex.  Therefore Condition~\ref{convexity} is a sufficient condition for the existence of an optimal strict control for the discounted control problem as well as for the long-term average problem.  This observation is codified in the following theorem.

\begin{thm} \label{thm:strict-control-disc}
Assume Conditions~\ref{gencnd}, \ref{costcnd1} and \ref{convexity} hold.  Let $c^*<\infty$ denote the optimal value of the linear program (\ref{disc-lp}).  Then there exists an optimal strict control $u^*:E\rightarrow U$.
\end{thm}

\section{Inventory Control} \label{sect:example}
In this section, we highlight the results using the inventory control problem of Example~\ref{inventory-example}.  Due to the presence of the fixed cost $k_1$ per order, continuous ordering would result instantly in an infinite cost so orders are placed at a discrete set of times.  Consequently, $\mu_0$ is a measure on the state space alone.  The function $c_0$ of (\ref{inv-run-cost}) satisfies the requirements of Condition~\ref{costcnd1} once the control dependence is eliminated.

The ordering cost function $c_1(x,u) = k_1 + k_2 u$ only depends on the amount ordered so does not satisfy the inf-compactness requirement in the state variable.  Inf-compactness is solely used to establish the tightness of sequences of expected ordering measures $\{\mu_1^n\}$ and hence the existence of limiting measures.  At least two adjustments are possible.  The first is to impose a budget constraint in which the function $h$ is inf-compact while a second approach is to restrict the class of admissible controls.  Each results in a modified problem.  

Intuitively it seems reasonable that an optimal ordering policy will not order when the inventory is large nor wait until there is a large back-order.  It also seems reasonable that the inventory immediately after an order is received should not be overly large.  Thus consider a restriction on the class of admissible policies such that $\{(X(\tau_k-),X(\tau_k)): k \in \NN\} \subset K$ for some compact set $K \in \RR^2$.  The compact set may depend on $(\tau,Y)$.

Set ${\cal D} = C^2_c(\RR)$.  The generator $A$ of the drifted Brownian motion is 
$$Af (x) = \mbox{$\frac{\sigma^2}{2}$} f''(x) - \mu f'(x), \qquad x\in \RR$$
while the singular generator is
$$Bf(x,u) = f(x+u) - f(x), \qquad x \in \RR, u \in \RR_+.$$
The long-term average problem of minimizing $J_0(\tau,Y)$ of (\ref{inv-lta-cost}) over the restricted class of admissible policies is equivalent to the linear program
$$\begin{array}{lll}
\mbox{Minimize } & \displaystyle \int_\RR c_0(x)\, \mu_0(dx) + \int_{\RR\times \RR_+} c_1(x,u)\, \mu_0(dx\times du) & \\
\mbox{Subject to } & \displaystyle \int_\RR Af(x)\, \mu_0(dx) + \int_{\RR\times \RR_+} Bf(x,u)\, \mu_0(dx\times du) = 0, & \quad f \in C_c^2(\RR), \\
& \mu_0 \in {\cal P}(\RR), & \\
& \mu_1 \in {\cal M}(\RR\times \RR_+) \mbox{ with compact support.} &
\end{array} $$
The discounted problem of minimizing $J_\alpha(\tau,Y)$ of (\ref{inv-disc-cost}) over this restricted class is equivalent to the linear program
$$\begin{array}{lll}
\mbox{Minimize } & \displaystyle \int_\RR c_0(x)\, \mu_0(dx) + \int_{\RR\times \RR_+} c_1(x,u)\, \mu_0(dx\times du) & \\
\mbox{Subject to } & \displaystyle \int_\RR A^\alpha f(x)\, \mu_0(dx) + \int_{\RR\times \RR_+} Bf(x,u)\, \mu_0(dx\times du) = 0, & \quad f \in C_c^2(\RR), \\
& \mu_0 \in {\cal P}(\RR), & \\
& \mu_1 \in {\cal M}(\RR\times \RR_+) \mbox{ with compact support} &
\end{array} $$
or, using the rescaling of Remark~\ref{rescaling}, to the linear program 
$$\begin{array}{ll}
\mbox{Minimize } & \displaystyle \int_\RR c_0(x)\, \mu_0(dx) + \int_{\RR\times \RR_+} c_1(x,u)\, \mu_0(dx\times du)  \\
\mbox{Subject to } & \displaystyle \int_\RR (Af(x)-\alpha f(x)\, \mu_0(dx) + \int_{\RR\times \RR_+} Bf(x,u)\, \mu_0(dx\times du) = -\int_\RR f(x)\, \nu_0(dx), \\
& \hfill f \in C_c^2(\RR), \\
& \mu_0 \in {\cal M}(\RR) \mbox{ with } \mu_0(\RR) = \frac{1}{\alpha}, \\
& \mu_1 \in {\cal M}(\RR\times \RR_+) \mbox{ with compact support} 
\end{array} $$
in which $\nu_0$ denotes the initial distribution of the inventory level.

The recent work \cite{helm:17} solves the long-term average inventory control problem in which the dynamics of the inventory level between orders are given by a general diffusion process for very general cost functions $c_0$ and $c_1$ without restricting the class of admissible ordering functions.  Interestingly, the approach is very similar to this paper but the existence of limiting $\mu_1$ measures is not required and hence an equivalent linear program is not utilized.

\end{document}